\documentclass{article}
\title{Dualities in Multiparameter Persistence}
\date{\today}
\author{%
	Ulrich Bauer\thanks{Department of Mathematics, TUM School of Computation, Information and Technology, Munich Data Science Institute (MDSI), and Munich Center for Machine Learning (MCML), Technical University of Munich (TUM), Germany},
	Fabian Lenzen\thanks{Institute of Mathematics, Technical University of Berlin, Germany},
	Michael Lesnick\thanks{Department of Mathematics, SUNY Albany, USA\\
		FL was funded by the Deutsche Forschungsgemeinschaft (DFG, German Research Foundation) under Germany's Excellence Strategy -- The Berlin Mathematics Research Center MATH+ (EXC-2046/1, project ID: 390685689).
		UB and FL were supported by the German Research Foundation (DFG) through the Collaborative Research Center SFB/TRR 109 Discretization in Geometry and Dynamics (project ID: 195170736).
		FL and ML also acknowledge the support of a grant from the Simons Foundation (Award ID 963845).
	}
}
\usepackage[left=1.6in, right=1.6in, top=1in, bottom=1.25in]{geometry}
\input{preamble}
 
\usepackage{refcount}
\begin{document}
\maketitle
\begin{abstract}
	In the theory of persistent homology,
	a well known duality relates the barcodes of the absolute homology and relative cohomology of a one-parameter simplicial filtration.
	Motivated by the problem of computing free presentations of the (co)homology of multiparameter Rips filtrations,
	we give a multiparameter generalization of this duality.  
	Considering two duality functors on multiparameter persistence modules,
	the \emph{pointwise dual} $(-)^*$ and the \emph{global dual} $(-)^\dagger$,
	we show that $H_q(C)^* \cong H^{N+q}(C^\dagger)$ for chain complexes
	$C$ of free $N$\-/parameter persistence modules with acyclic colimit.
We give an elementary and accessible proof based on a long exact sequence argument, and also give an alternate proof that casts the result as a special case of multigraded Grothendieck local duality.  As a corollary,
	we recover a simple correspondence between minimal free resolutions of a persistence module $M$ and those of its pointwise dual $M^*$, 
	a result previously obtained by \textcite[Theorem~4.5(3)]{Miller:2000}.
	These results form the foundation of a state-of-the-art algorithm for computing free resolutions of the homology of Vietoris--Rips bifiltrations, 
	described in a forthcoming paper.
\end{abstract}

\tableofcontents

\section{Introduction}
Persistent homology provides computable invariants of filtered simplicial complexes called \emph{barcodes}:
Given a $\Z$\-/indexed filtration
\[
K\colon \dotsb \into K_{-1} \into K_0 \into K_1 \into\dotsb
\]
of a finite simplicial complex $\cup K \coloneqq \bigcup_z K_z$,
a standard structure theorem \cite{BotnanWCB:2020} tells us
that $H_d(K)$, the $d$th reduced homology of $K$ with coefficients in a fixed field $\FF$,
decomposes uniquely (up to isomorphism) as a direct sum
\begin{equation}
	\label{eq:barcode-decomp-homology}
	H_d(K) \cong \smashoperator{\bigoplus_{I \in \barc H_d(K)}} \FF_{I}.
\end{equation}
of \emph{interval modules} $\FF_{I}$,
indexed by a multiset $\barc H_d(K)$ of intervals with endpoints in $\Z\cup\{\pm\infty\}$.
The multiset $\barc H_d(K)$ is called the \emph{barcode} of $H_d(K)$ \cite{Webb:1985,ZomorodianCarlsson:2005}.  
The classical algorithm for computing these barcodes \cite{EdelsbrunnerLetscherEtAl:2002,ZomorodianCarlsson:2005,EdelsbrunnerHarer:2008}
applies a column reduction scheme, similar to Gaussian elimination, to the boundary matrices of the chain complex $C_\bullet(K)$ of $K$.

Besides the absolute homology $H_d(K)$, one may also consider the absolute cohomology $H^d(K)$ of $K$,
as well as the relative homology $H_d(\cup K, K)$ and relative cohomology $H^d(\cup K, K)$,
where, as throughout the paper, we work with reduced (co)homology of filtrations with coefficients in $\FF$.
The same structure theorem yields barcodes for each of these, and these four classes of barcodes determine each other in a simple way \cite{deSilvaMorozovEtAl:2011,BauerSchmahl:2021a}.
The correspondence between the absolute homology and relative cohomology turns out to be important in particular for persistent homology computation.
The precise statement of this correspondence is the following:

\begin{theorem}[{\cites[Propositions~2.3 and~2.4]{deSilvaMorozovEtAl:2011}[Theorem 6.2]{BauerSchmahl:2021a}}]
	\label{thm:dSMVJ-detail}
	If $K$ is a finite $\Z$\-/indexed simplicial filtration, then there is a bijective correspondence between persistence barcodes of $H_\bullet(K)$ and $H^\bullet(\cup K, K)$ such that
	\begin{itemize}
		\item
		each bounded interval $I$ in $\barc H_d(K)$ corresponds to the bounded interval $-I$ in $\barc H^{d+1}(\cup K, K)$, and
		\item
		each unbounded interval $J$ in $\barc H_d(K)$ corresponds to the unbounded interval $\Z \setminus (-J)$ in $\barc H^{d}(\cup K, K)$.
	\end{itemize}
\end{theorem}

When $\cup K$ is acyclic, this correspondence can be expressed as a natural isomorphism of functors, as in the following theorem, 
which follows easily from the long exact sequence of the pair $(\cup K, K)$.   
\begin{theorem}[{\cite[Theorem 6.2]{BauerSchmahl:2021a}}]
	\label{thm:dSMVJ}
	If $K$ is a finite $\Z$\-/indexed simplicial filtration with $H_d(\cup K) = H_{d+1}(\cup K) = 0$, then
	there is a natural isomorphism of $\Z$\-/persistence modules
	\[
		H_d(K)^* \cong H^{d+1}(\cup K, K).
	\]
\end{theorem}
Here, $(-)^*$ denotes the pointwise dual of a persistence module; see \cref{def:pointwise-dual}.
\cref{thm:dSMVJ,thm:dSMVJ-detail} also hold as given for unreduced (co)homology, but the reduced case is more relevant to the main results of this paper.

\Cref{thm:dSMVJ-detail} implies that $\barc H_d(K)$ can be obtained by computing the barcodes of $H^{d}(\cup K, K)$ and $H^{d+1}(\cup K, K)$.
This leads to a highly efficient algorithm for computing the persistence barcodes of the Vietoris--Rips filtration of a finite metric space.
To explain, recall that the classical matrix reduction algorithm for computing persistent homology \cite{ZomorodianCarlsson:2005} takes as input any (co)chain complex of free graded $\FF[x]$\-/modules.
Since the absolute filtered chain complex $C_\bullet(K)$ and the relative filtered cochain complex $C^\bullet(\cup K, K)$ associated to $K$ are both filtered complexes of this type,
the standard algorithm can be applied either to $C_\bullet(K)$, yielding $\barc H_\bullet(K)$, or to $C^\bullet(\cup K, K)$, yielding $\barc H^\bullet(\cup K, K)$.
For $K$ a Vietoris--Rips filtration, the latter turns out to be orders of magnitude faster than the former when used together with a simple optimization of the reduction algorithm called \emph{clearing} \cite{ChenKerber:2011}.
This approach is essential to an efficient implementation of Vietoris--Rips persistence computation \cite{dionysus2,GUDHI,Bauer2017Phat,Henselman-Petrusek:2021,Bauer:2019}.

In contrast to the one-parameter setting, there are algebraic obstacles to defining the barcode of the homology of a multiparameter filtration \cite{CarlssonZomorodian:2009}.
Therefore, in practice one works with the homology of a multiparameter filtration via \emph{incomplete invariants}
such as the Hilbert function, rank invariant, or multigraded Betti numbers, among others \cite{CarlssonZomorodian:2009,BotnanLesnick:2023}.
To compute such invariants, it is often most efficient to first compute a \emph{minimal free presentation} of multiparameter persistent homology.
For the case of two-parameter persistence, recent work by \textcite{LesnickWright:2022,FugacciKerberEtAl:2023} introduced a cubic-time algorithm
to compute a minimal free presentation (or resolution) of the homology
of a simplicial bifiltration $K$
(or of any chain complex of free bipersistence modules),
which performs well enough for practical data analysis.
The algorithm is a matrix reduction scheme similar to the standard algorithm for computing persistent homology.

In analogy with the one-parameter case, one might hope to expedite the computation
of minimal presentations or resolutions of the homology of a sublevel-Rips bifiltration $K$ \cite{BotnanLesnick:2023}
by applying the algorithm of~\cite{LesnickWright:2022,FugacciKerberEtAl:2023} to the relative cochain complex $C^\bullet(\cup K, K)$.
However, this cannot be done naively because, while the natural isomorphism from \cref{thm:dSMVJ} exists for any number of persistence parameters,
the cochain modules of $C^\bullet(\cup K, K)$ are not free for more than one parameter.

To circumvent this problem, the natural first step is to identify a duality result
for (co)chain complexes of free two-parameter persistence modules which extends \cref{thm:dSMVJ};
this is the main goal of the present paper.
In fact, this duality result holds for an arbitrary number of parameters.
In a companion paper \cite{companion-computation}, we apply the two-parameter case of this duality
to derive an algorithm for efficiently computing minimal resolutions of sublevel-Rips bifiltrations.
We hope that in the future, our duality result can also enable more efficient resolution computation for other types of Vietoris--Rips type multifiltrations, e.g.,
degree-Rips bifiltrations \cite{LesnickWright:2015,BlumbergLesnick:2022a}, interlevel-Rips trifiltrations \cite{BotnanLesnick:2023},
and Rips trifiltrations of dynamic metric spaces \cite{kim2021spatiotemporal}.

%To elaborate, in the ungraded setting, while many references state only an underived form of Grothendieck
%local duality, the derived formulation is also classical, dating back to \textcite[278]{Hartshorne:1966}.
%In the multigraded setting, the underived form of local duality is given by \textcite[Theorem~2.2.2]{GotoWatanabe:1978a},
%and though not explicitly stated, the derived form is implicit in their proof.
%Further, a non-local generalization of (the derived form of) Grothendieck local duality, 
%called \emph{Greenlees--May Duality}, is well known, and been adapted to the multigraded setting in the work of \textcite[Theorem~5.3]{Miller:2002}.  

\subsubsection*{Statements of Results}
To extend \cref{thm:dSMVJ} to the multiparameter setting, we consider two contravariant endofunctors $(-)^*$ and $(-)^\dagger$
on the category of $\Z^N$\-/persistence modules, which we call the \emph{pointwise} (or \emph{Matlis}) and \emph{global duality functors}.
Since both functors are contravariant, they send chain complexes of $\Z^N$\-/persistence modules
to cochain complexes of $\Z^N$\-/persistence modules and vice versa.
In the following, $C$ denotes a chain complex of free $\Z^N$\-/persistence modules.
Such a chain complex $C$ could, for example, arise from a $\Z^N$\-/indexed filtration of a simplicial complex.  
However, the following results are purely algebraic and do not depend on simplicial filtrations.
A  $\Z^N$\-/persistence module $M$ \emph{eventually vanishes} if the \emph{support} $\supp M \coloneqq \Set{z \in \Z^N ; M_z \neq 0}$ of $M$ is bounded above; see \cref{def:eventually-acyclic}.
Our main duality result is the following:

\begin{theorem}[{see \cpageref{proof:local-duality}}]\label{thm:local-duality}
	If $C$ is a chain complex of finite rank free $\Z^N$\-/persistence modules
	such that the modules $H_d(C),\dotsc,H_{d+N}(C)$ eventually vanish,
	then there is a natural isomorphism
	\[
		H_d(C)^* \cong H^{d+N}(C^\dagger).
	\]
\end{theorem}
If $N = 1$, we have $C_\bullet(K)^\dagger \cong C^\bullet(\cup K, K)$; see \cref{rmk:dagger-in-1D}.
Therefore, we may view \cref{thm:local-duality} as a multiparameter extension of \cref{thm:dSMVJ}.

We obtain \cref{thm:local-duality} by proving a slightly stronger derived form of the result, \cref{thm:main},
for which we give a self-contained, elementary proof, using a long exact sequence argument.
In \cref{sec:local-duality}, we also give an alternate proof by casting the statement as special case of a multigraded, derived form 
of Grothendieck local duality.
We thank Ezra Miller for explaining this alternate proof to us.
The form of multigraded local duality required here is implicit in the classical literature~\cite[Theorem~2.2.2]{GotoWatanabe:1978a},
and also has appeared explicitly in greater generality, as \emph{multigraded Greenlees--May Duality} \cite[Theorem~5.3]{Miller:2002}.
While the alternate proof is arguably not much simpler than our direct, elementary argument,
it clarifies the connection between our result and classical duality ideas in commutative algebra.

As a corollary of \cref{thm:local-duality}, we recover a duality result for free resolutions of an eventually vanishing $\Z^N$\-/persistence module, which previously appeared in work of \textcite[Theorem~4.5(3)]{Miller:2000}.
To state this result, recall that a \emph{free resolution} of a persistence module $M$
is a chain complex $G$ of free modules concentrated in nonnegative degrees, such that $H_0(G) \cong M$ and $H_d(G) = 0$ for $d > 0$.
Note that we can regard a chain complex $C$ as a cochain complex with components $C^d = C_{-d}$.
For a (co)chain complex $C$, define the \emph{shifted complex} $C[N]$ with components $C[N]_d=C_{d+N}$ or $C[N]^d = C^{d+N}$, respectively.

\begin{corollary}[{see \cpageref{proof:module-resolutions}}]\label{thm:module-resolutions}
	Let $M$ be a finitely generated, eventually vanishing $\Z^N$\-/persistence module.
	Then
	\begin{enumerate}
		\item\label{thm:module-resolutions:at-least-N}
			any free resolution of $M$ has length at least $N$, with equality for \MFR{}s,
		\item\label{thm:module-resolutions:dual}
			if $G$ is a (minimal) free resolution of $M$ with length $N$, then $G[N]{}^\dagger$, viewed as a chain complex, is a (minimal) free resolution of $M^*$.
	\end{enumerate}
\end{corollary}

\Textcite{Miller:2000} observed that \cref{thm:module-resolutions} is a special case of multigraded Grothendieck local duality.
We explain this in \cref{sec:local-duality}.  
%However, our proof of \cref{thm:module-resolutions} via \cref{thm:local-duality} is more elementary.

Combining \cref{thm:local-duality,thm:module-resolutions}, we obtain:

\begin{corollary}
	\label{Cor:Combined_Result}
	If $C$ is as in \cref{thm:local-duality} and
	$G$ is a minimal free resolution of $H^{d+N}(C^\dagger)$,
	then $G[N]{}^\dagger$ is a minimal free resolution of $H_d(C)$.
\end{corollary}

In the computational applications that motivate our results, the chain complexes we encounter
often do not have eventually vanishing homology.
For example, if $K$ is a function-Rips bifiltration, then the colimit of $H_d(C_\bullet(K))$ along the function parameter need not be zero.
Fortunately, our results extend readily to versions which do not require eventually vanishing homology:
In \cref{sec:cones},
given a chain complex $C$ of free $N$\-/parameter persistence modules of finite total rank,
we construct a chain complex $\hat C$ with eventually vanishing homology such that $C$ is the restriction of $\hat{C}$ along a certain poset map $r\colon \Z^N\to \Z^N$.
The construction amounts to coning off $C$ at sufficiently large $\Z^N$\-/indices.
We can then extend our duality results to $C$ by applying them to $\hat C$ and restricting along $r$.
In particular, letting $R$ denote the restriction along $r$, we obtain the following variant of \cref{Cor:Combined_Result},
which requires no eventual vanishing assumption on homology.

\begin{corollary}[see \cref{proof:Combined_Result_Extended}]
	\label{Cor:Combined_Result_Extended}
	If $C$ is a free chain complex of finite total rank and
	$G$ is a minimal free resolution of $H^{d+N}(\hat C^\dagger)$,
	then $R(G[N]{}^\dagger)$ is a minimal free resolution of $H_d(C)$.
\end{corollary}
In our companion paper \cite{companion-computation}, we apply \cref{Cor:Combined_Result_Extended} to develop an efficient cohomological algorithm
to compute a minimal free resolution of $H_d(C)$, for $C=C_\bullet(K)$ the simplicial chain complex of a finite simplicial bifiltration $K$.  %This algorithm performs competitively on Rips-like bifiltrations.

The present paper and its forthcoming companion \cite{companion-computation} both extend results that appeared
in an earlier conference paper \cite{BauerLenzenEtAl:2023a}, with substantially improved exposition.

\subsubsection*{Outline}
We summarize the necessary background on persistence modules and resolutions in \cref{Sec:Pers_Mods}.
In \cref{sec:dualities}, we prove \cref{thm:main}, which immediately implies our main results \cref{thm:local-duality,thm:module-resolutions}.
For completeness, we also state a dual version of \cref{thm:main} in \cref{thm:main:dual}.  
We conclude the section with some examples, which, in particular, illustrate how our statements generalize \cref{thm:dSMVJ} to the multi\-/parameter case.
Whereas the proof of \cref{thm:main} in \cref{sec:dualities} is an elementary long exact sequence argument, in
\cref{sec:local-duality}, we provide our alternate proof via Grothendieck local duality.  
To make the proof accessible, we include a brief recap of the relevant notions from commutative algebra and the theory of derived functors.
To conclude the paper, we construct $\hat C$, the replacement with vanishing homology of a free chain complex $C$, and we prove \cref{Cor:Combined_Result_Extended}.

\subsubsection*{Acknowledgements}
We thank Ezra Miller for sharing helpful comments on earlier versions of this paper, 
and especially for explaining to us how our main duality is related to Grothendieck local duality and Greenlees--May duality.
We also thank Steffen Oppermann for helpful discussions about dualities in representation theory,
and the reviewers of the earlier conference version of this work \cite{BauerLenzenEtAl:2023a} for valuable input.

\section{Background}
In this section, we review the relevant background on persistence modules,
commutative and homological algebra.
Most of the material in this section is standard; e.g., see \textcite{Miller:2000}, where many of the same definitions appear.

\subsection{Persistence modules}\label{Sec:Pers_Mods}
Fix a field $\FF$, and let $\vect$ denote the category of finite-dimensional $\FF$\-/vector spaces.
For $N \in \N$, we consider the poset $\Z^N$ with the product partial order.
The poset $\Z^N$ is a lattice, where for $z, z' \in \Z^N$, we write $z \vee z'$ and $z \wedge z'$
for the componentwise maximum and minimum, respectively.

A \emph{$\Z^N$\-/persistence module} is a functor $M\colon \Z^N \to \vect$, $z \mapsto M_z$, $(z \leq z') \mapsto M_{z'z}$.
We call the vector spaces $M_z$ \emph{components} and the morphisms $M_{z'z}\colon M_z \to M_{z'}$ \emph{structure maps} of $M$.

A morphism $\gamma\colon L\to M$ of $\Z^N$\-/persistence modules is a natural transformation of functors; we write $\Hom(L, M)$ for the vector space of all morphisms from $L$ to $M$.
With these morphisms, the $\Z^N$\-/persistence modules form a category, which we denote as $\pers{\Z^N}$.
It is an abelian category, where (co)kernels, images and finite direct sums are given componentwise.

\begin{remark}
	\label{rmk:eq-pers-mod}
	The category $\pers{\Z^N}$ is equivalent to the category of $\Z^N$\-/graded $\FF[x_1, \dotsc, x_N]$\-/modules
	with finite-dimensional homogeneous components \cite{CarlssonZomorodian:2009}.
\end{remark}

\begin{figure}
	\newcommand{\SetupDiagram}{
		\Axes[->](-1.5,-1.5)(3.5,3.5)
		\Grid[-1](0){3}[-1](0){3}
	}
	\tikzcdset{every diagram/.append style={row sep=.8em, column sep=.9em, cramped,cells={}}}
	\tikzset{every module diagram/.append style={x=.45cm, y=.45cm, baseline={([yshift=-1ex]current bounding box.center)}}}
	\scriptsize
	{}\hfill
	$\begin{tikzcd}
		0 \rar     & 0 \rar                            & 0   \rar                                              & 0 \rar              & 0      \\
		0 \uar\rar & \FF \uar\rar[equal]               & \FF   \uar\rar[equal]                                 & \FF \uar\rar        & 0 \uar \\
		0 \uar\rar & \FF \uar[equal]\rar["\Mtx{0\\1}"] & \FF^2 \uar\rar\ar[ur, phantom, "\scriptstyle{(1,1)}"] & \FF \uar[equal]\rar & 0 \uar \\
		0 \uar\rar & 0 \uar\rar                        & \FF   \uar["\Mtx{1\\0}"']\rar[equal]                  & \FF \uar[equal]\rar & 0 \uar \\
		0 \uar\rar & 0 \uar\rar                        & 0 \uar\rar                                            & 0 \uar\rar          & 0 \uar
	\end{tikzcd}
	=
	\begin{tikzpicture}[module diagram]
		\SetupDiagram
		\begin{scope}[blue]
			\filldraw[fill opacity=0.2] (.5,-.5) -| (2.5,2.5) -| (-.5,.5) -| cycle;
			\filldraw[fill opacity=0.2] (.5,.5) rectangle (1.5,1.5);
			\draw[->] (1,.25) to["$\Mtx{1 \\ 0}$" {right, near start, anchor=north west, inner sep=1pt}] (1,.75);
			\draw[->] (.25,1) to["$\Mtx{0 \\ 1}$" {above, near start, anchor=-55, inner sep=1pt}] (.75,1);
			\draw[->] (1.25,1.25) to["$\scriptstyle (1\:1)$" {left, at end, anchor=south east, inner sep=1pt}] (1.75,1.75);
		\end{scope}
	\end{tikzpicture}$
	\hfill
	$\begin{tikzcd}
		                   & \vdots                            & \vdots                               & \vdots                       &        \\
		\cdots \rar[equal] & \FF \uar[equal]\rar               & \FF^2   \uar[equal]\rar[equal]       & \FF^2 \uar[equal]\rar[equal] & \cdots \\
		\cdots \rar[equal] & \FF \uar[equal]\rar["\Mtx{0\\1}"] & \FF^2 \uar[equal]\rar[equal]         & \FF^2 \uar[equal]\rar[equal] & \cdots \\
		\cdots \rar[equal] & 0 \uar\rar                        & \FF   \uar["\Mtx{1\\0}"']\rar[equal] & \FF \uar\rar[equal]          & \cdots \\
		                   & \vdots\uar[equal]                 & \vdots\uar[equal]                    & \vdots\uar[equal]            &
	\end{tikzcd}
	=
	\begin{tikzpicture}[module diagram]
		\SetupDiagram
		\coordinate (g) at (1,1);
		\FreeModule<x>[blue](g){}
		\FreeModule<y>[blue](g){}
	\end{tikzpicture}$
	\hfill{}
	\caption{Illustrations of $\Z^2$\-/persistence modules $M$.
	Where necessary, a matrix representing the structure map $M_z \to M_{z+e_i}$ w.r.t.\ some basis is shown.
		If the structure map is equality, or a canonical inclusion or projection, then the matrix is omitted.
		In the diagrams, each grid point corresponds to an element $z \in \Z^2$.
		Intensity of shading indicates the value of $\dim M_z$, with no shading over points $z$ such that $\dim M_z=0$.
		Borders between regions indicate changes in pointwise dimension.}
	\label{fig:module-illustrations}
\end{figure}
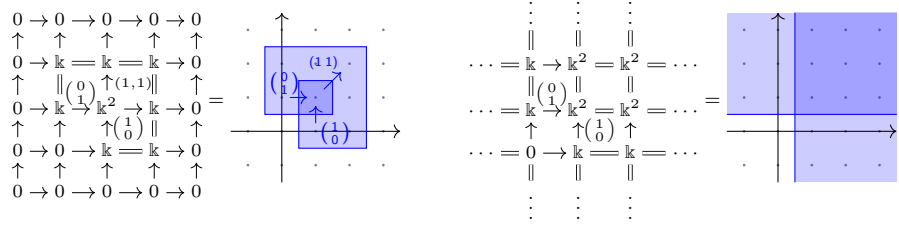
To illustrate our statements and proofs, we use sketches of $\Z^2$\-/persistence modules, as in \cref{fig:module-illustrations}.

A persistence module $M$ is \emph{finitely generated} if there are $s_1 \in M_{g_1},\dotsc,s_r \in M_{g_r}$
such that for each $z \in \Z^N$, the set $\Set{M_{z,g_i}(s_i); g_i \leq z}$ spans the vector space $M_z$.

For $z \in \Z^N$, the \emph{shift} of a persistence module $M$ by $z$ is the module $M\Shift{z}$ with components $M\Shift{z}_w = M_{z+w}$.
If $z \geq 0$, the structure maps of $M$ induce a morphism $M \to M\Shift{z}$.

\begin{definition}
	\label{def:three-functors}
	Let $L, M \in \pers{\Z^N}$.
	We define the following persistence modules,
	whose structure maps are induced by the ones of $L$ and $M$:
	\begin{enumerate}
		\item \label{def:pointwise-dual}
		The \emph{pointwise} or \emph{Matlis dual} $M^*$ of $M$ is the $\Z^N$\-/persistence module with components $(M^*)_z = (M_{-z})^*$ and structure maps $(M^*)_{z'z} = (M_{-z', -z})^*$,
		where $V^* = \Hom_\FF(V, \FF)$ denotes the dual of a $\FF$\-/vector space $V$.

		\item $\IHom(L, M)$ is the $\Z^N$\-/persistence module with components
		\[
			\IHom(L, M)_z = \Hom(L, M\Shift{z}) = \Hom(L\Shift{-z}, M).
		\]

		\item
		The \emph{tensor product} $L \otimes M$ is the $\Z^N$\-/persistence module with components
		\[
			(L \otimes M)_z = \Bigl(\smashoperator[r]{\bigoplus_{x + y = z}} L_{x} \otimes_\FF M_{y}\Bigr) /{\sim},
		\]
		where
		\(
			L_{x+w,x}(l) \otimes m \sim l \otimes M_{y+w,y}(m)
		\)
		for all $x, y \in \Z^N$, $l \in L_{x}$, $m \in M_{y}$ and $w\geq 0$.
	\end{enumerate}
\end{definition}

These definitions extend to
an exact functor
\[
	(-)^*\colon (\pers{\Z^N})^\op \to \pers{\Z^N},
\]
a left exact bifunctor
\[
	\IHom\colon (\pers{\Z^N})^\op \times \pers{\Z^N} \to \pers{\Z^N},
                       \]
and
a right exact bifunctor
\[
	\otimes\colon \pers{\Z^N} \times \pers{\Z^N} \to \pers{\Z^N}.
\]
Recall that left or right exactness of a functor between abelian categories implies that the functor is additive, i.e., it preserves finite direct sums.
For any $M$, the functors $M \otimes -$ and $\IHom(M, -)$ are adjoint.
Since we only consider persistence modules with finite-dimensional components, $(-)^*$ is an equivalence;
in particular, the canonical inclusion
\begin{equation}
	\label{eq:dual-iso}
	\IHom(L, M) \to \IHom(M^*, L^*)
\end{equation}
is an isomorphism.

Let $\uZ \coloneqq \Z \cup \{\infty\}$ and $\lZ \coloneqq \Z \cup \{-\infty\}$.
For $z \in \lZ^N$ and $z' \in \uZ^N$,
let $F(z)$ and $I(z')$ be the persistence modules with
\begin{align*}
	F(z)_{v} &= \begin{cases}
		\FF & \text{if $z \leq v$}, \\
		0   & \text{otherwise,}
	\end{cases} &
	I(z')_{v} &= \begin{cases}
		\FF & \text{if $v \leq z'$}, \\
		0   & \text{otherwise,}
	\end{cases}
\end{align*}
and identities as structure maps between nonzero components.
A persistence module $M$
is \emph{flat} if there is an isomorphism $b\colon \bigoplus_{z \in R} F(z) \to M$ for a multiset $R \subseteq \lZ^N$,
and \emph{injective} if there is an isomorphism $b\colon \bigoplus_{z \in R} I(z)\to M$ for $R \subseteq \uZ^N$.
In either case, the multiset $R$ is uniquely determined by $M$ and called the \emph{graded rank of $M$}, denoted by $\rk M$.
The isomorphism $b$ is a \emph{generalized basis of $M$}.
A flat (respectively, injective) module $M$ is \emph{free} (respectively, \emph{cofree}) if $\rk M \subseteq \Z^N$;
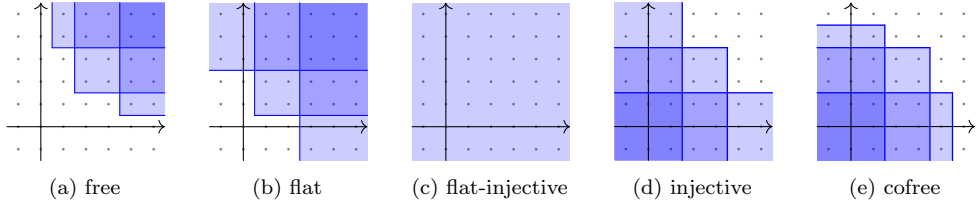
\begin{figure}
	\newcommand\Setup{
		\Axes(-1.5,-1.5)(5.5,5.5)
		\Grid[-1](0){5}[-1](0){5}
		\coordinate (g) at (3,3);
	}
	\begin{subfigure}{.2\linewidth}
		\centering
		\begin{tikzpicture}[module diagram]
			\Setup
			\FreeModule[blue](2,2){}
			\FreeModule[blue](1,4){}
			\FreeModule[blue](4,1){}
		\end{tikzpicture}
		\caption{free}
	\end{subfigure}%
	\begin{subfigure}{.2\linewidth}
		\centering
		\begin{tikzpicture}[module diagram]
			\Setup
			\FreeModule<x>[blue](g){}
			\FreeModule<y>[blue](g){}
			\FreeModule[blue](1,1){}
		\end{tikzpicture}
		\caption{flat} % = projective
	\end{subfigure}%
	\begin{subfigure}{.2\linewidth}
		\centering
		\begin{tikzpicture}[module diagram]
			\Setup
			\FreeModule<xy>[blue](g){}
		\end{tikzpicture}
		\caption{flat-injective}
	\end{subfigure}%
	\begin{subfigure}{.2\linewidth}
		\centering
		\begin{tikzpicture}[module diagram]
			\Setup
			\coordinate (g) at (1,1);
			\InjModule<x>[blue](g){}
			\InjModule<y>[blue](g){}
			\InjModule[blue](3,3){}
		\end{tikzpicture}
		\caption{injective}
	\end{subfigure}%
	\begin{subfigure}{.2\linewidth}
		\centering
		\begin{tikzpicture}[module diagram]
			\Setup
			\InjModule[blue](4,1){}
			\InjModule[blue](3,3){}
			\InjModule[blue](1,4){}
		\end{tikzpicture}
		\caption{cofree}
	\end{subfigure}
	\caption{Free, flat (equivalently, projective), injective and cofree modules.
		Each quadrant, half plane or entire plane corresponds to one free, flat, injective or cofree indecomposable summand.}
	\label{fig:free-flat-inj-modules}
\end{figure}
see \cref{fig:free-flat-inj-modules}.
Note that $M$ is flat if and only if $M^*$ is injective, and $M$ is free if and only if $M^*$ is cofree.

These definitions of flat and injective modules agree with the usual definitions,
i.e., $M$ is flat (resp.\ injective) if and only if the functor $M \otimes -$ (resp.\ $\IHom(M,-)$) is exact \cite[Lemma~11.23, Theorem~11.30]{MillerSturmfels:2005}.
Similarly, for free modules the definition of a generalized basis is equivalent to the usual definition of basis.

\begin{remark}
	\label{rmk:projectives-in-pers}
	Write $\Pers{\Z^N}$ for the category of persistence modules $M$ without restriction on $\dim M_z$.
	While the flat and the injective modules in $\pers{\Z^N}$ are precisely 
	the flat and injective modules of $\Pers{\Z^N}$, respectively, that lie in $\pers{\Z^N}$,
	the same is not true for projective modules.
	In particular, it is not true that all projective modules in $\pers{\Z^N}$ are direct summands of flat modules
	(as is the case in $\Pers{\Z^N}$).
	Instead, it follows from \eqref{eq:dual-iso} and the classification of injective modules in \cite[Theorem~11.30]{MillerSturmfels:2005}
	that the projective objects in $\pers{\Z^N}$ coincide with the flat modules.
\end{remark}

\begin{remark}
	Via the equivalence of categories $\Pers{\Z^N} \simeq \gMod{R}$ for $R = \FF[x_1,\dotsc,x_N]$,
	an indecomposable flat module $F(z)$ for $z \in \lZ^N$ corresponds to the localization $S^{-1} R\Shift{\tilde{z}}$
	for the multiplicative system $S$ generated by $\Set{x_i; z_i = -\infty}$,
	and $\tilde{z} \in \Z^N$ the vector with $\tilde{z}_i = z_i$ for $z_i$ finite and $\tilde{z}_i = 0$ otherwise.
	Analogously, an injective module $I(z)$ for $z \in \uZ^N$ corresponds to the injective hull $E(R/\mathfrak{p})\Shift{\tilde{z}}$
	for $\mathfrak{p} = (x_i)_{z_i \neq \infty}$.
\end{remark}

\subsection{Graded matrices}
\label{sec:graded-matrices}
A \emph{$\Z^N$\-/graded $m\times n$\-/matrix} $\Mat{M}$ consists of an underlying matrix $\uMat{\Mat{M}} \in \FF^{m \times n}$,
together with row grades $\rg^{\Mat{M}} \in (\Z^N)^m$ and column grades $\cg^{\Mat{M}} \in (\Z^N)^n$.
More generally, $\lZ^N$\-/graded and $\uZ^N$\-/graded matrices are defined in the same way.
Letting $\bm{r} = \rg^\Mat{M}$ and $\bm{c} = \cg^\Mat{M}$, we also call $\Mat{M}$ a $\bm{r} \times \bm{c}$\-/matrix.
We write $\FF^{\bm{r}\times\bm{c}}$ for the vector space of $\bm{r}\times\bm{c}$\-/matrices.
A graded matrix is \emph{valid} if $\Mat{M}_{ij} = 0$ whenever $\rg^\Mat{M}_i \not\leq \cg^\Mat{M}_j$.
Valid graded matrices are called \emph{monomial matrices} in \cite{Miller:2000}.
%In \cite[Definition~1.23]{MillerSturmfels:2005}, a valid graded matrix is called \emph{monomial matrix}.
We have
\begin{align*}
	\Hom(F(z), F(z')) &\cong \begin{cases*}
		\FF & if $z' \leq z$,\\
		0   & otherwise,
	\end{cases*}
	&&\text{for $z, z' \in \lZ^N$, and}\\
	\Hom(I(z), I(z')) &\cong \begin{cases*}
		\FF & if $z' \leq z$,\\
		0   & otherwise
	\end{cases*}
	&&\text{for $z, z' \in \uZ^N$,}
\end{align*}
where $\lambda \in \FF$ is identified with the linear map $x\mapsto \lambda x$ in components.
Thus, for vectors $\bm{r}$ and $\bm{c}$ with entries in $\lZ^N$, we have
\[
	\Hom\bigl(\toplus_j F(c_j), \toplus_i F(r_i)\bigr) \cong \Set{\Mat{M} \in \FF^{\bm{r}\times\bm{c}}; \text{$\Mat{M}$ is valid}}.
\]
Similarly, for vectors 	$\bm{r}'$ and $\bm{c}'$ with entries in $\uZ^N$, we have
\[
	\Hom\bigl(\toplus_j I(c'_j), \toplus_i I(r'_i)\bigr) \cong \Set{\Mat{M} \in \FF^{\bm{r}'\times\bm{c}'}; \text{$\Mat{M}$ is valid}}.
\]
Thus, if $L$ and $M$ are both flat (resp.\ free, cofree or injective) modules of finite generalized rank, then after choosing ordered generalized bases for $L$ and $M$,
we may identify morphisms from $L$ to $M$ with valid graded matrices.

\begin{example}
	Let $\bm{r} = \bigl((0,1), (1,0)\bigr)$ and $\bm{c} = \bigl((0,2), (1,1), (2,0)\bigr)$.
	Then
	\[
		\textstyle
		\Hom\bigl(\smashoperator{\toplus_{g\in\bm{c}}} F(g), \smashoperator{\toplus_{g\in\bm{r}}} F(g)\bigr)
		\cong \Hom\bigl(\smashoperator{\toplus_{g\in\bm{c}}} I(g), \smashoperator{\toplus_{g\in\bm{r}}} I(g)\bigr)
		= \biggl\{ \,
			\begin{pNiceMatrix}[small,first-col,first-row]
				& (0,2) & (1,1) & (2,0) \\
				(0,1) & * & * & 0 \\
				(1,0) & 0 & * & *
			\end{pNiceMatrix}
		\biggr\},
	\]
	where the asterisks stand for arbitrary entries from $\FF$.
\end{example}

\subsection{Chain complexes and homology}
If $C$ is a chain complex of persistence modules, its $d$\-/cycles $Z_d(C)$, $d$\-/boundaries $B_d(C)$ and $d$\-/homology $H_d(C)$ are persistence modules for every $d$.
Similarly, if $C$ is a cochain complex of persistence modules, then the $d$\-/cocycles $Z^d(C)$, $d$\-/coboundaries $B^d(C)$ and $d$\-/cohomology $H^d(C)$ are persistence modules for every $d$.
Every contravariant endofunctor $(\pers{\Z^N})^\op \to \pers{\Z^N}$ turns chain complexes into cochain complexes and vice versa;
in particular, this is the case for the pointwise dual $(-)^*$.
Since $(-)^*$ is an exact functor, for all $d$ there is a natural isomorphism
\begin{equation}
	\label{eq:uct}
	H^d(C^*) \stackrel\cong\to H_d(C)^*
\end{equation}
for every chain complex $C$ of persistence modules.

Let $\simp$ denote the category of finite simplicial complexes, and let $N \in \N$.
A \emph{simplicial $\Z^N$\-/filtration}
is a functor $K\colon \Z^N \to \simp$, $z \mapsto K_z$, $(z \leq z') \mapsto K_{z'z}$
such that all morphisms $K_{z'z}\colon K_z \to K_{z'}$ are inclusions.
We call $K$ \emph{finite} if $\cup K \coloneqq \bigcup_z K_z$ is finite.
The filtration $K$ is called \emph{one-critical} if for every $\sigma \in \cup K$,
the set $\Set{z; \sigma \in K_z}$ has a unique minimal element $g(\sigma)$.
For each simplicial complex $K_z$, let $K_z^{(d)}$ denote the set of $d$\-/simplices in $K_z$.

As noted in the introduction, we always work with reduced (co)homology of simplicial complexes:
if $L$ is a nonempty simplicial complex, we write $C_\bullet(L)$ for the augmented chain complex of $L$,
where $C_d(L)$ is the $\FF$\-/vector space with basis $L^{(d)}$ for $d \geq 0$,
$C_{-1}(L) = \FF$ and $C_{d}(L) = 0$ for $d < -1$.
If $K$ is a nonempty simplicial $\Z^N$\-/filtration,
we define the \emph{absolute filtered chain complex} $C_\bullet(K)$ of $K$ with $C_d(K)_z = C_d(K_z)$,
and the \emph{absolute filtered cochain complex} of $K$ by $C^\bullet(K) \coloneqq C_\bullet(K)^*$.
Further, if $\cup K$ is finite, we define the \emph{relative filtered chain complex} $C_\bullet(\cup K, K)$
with $C_d(\cup K, K)_z = C_d(\cup K)/C_d(K_z)$,
and the \emph{relative filtered cochain complex} of $K$ by $C^\bullet(\cup K, K) \coloneqq C_\bullet(\cup K, K)^*$.
The respective (co)homology persistence modules $H_d(K)$, $H^d(K)$, $H_d(\cup K, K)$ and $H^d(\cup K, K)$
are called the \emph{absolute} and \emph{relative persistent (co)homology} of $K$.

\begin{example}
	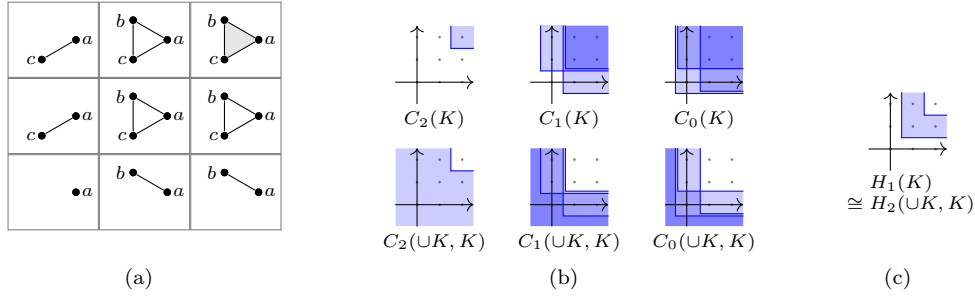
\begin{figure}
		\subcaptionbox{\label{fig:filtration:K}}{
				\raisebox{1ex}{
				\begin{tikzpicture}[
					Filtration/.pic={
						\begin{scope}[every node/.append style={fill, shape=circle, inner sep=1pt, font=\scriptsize}, every label/.append style={inner sep=0ex}]
							\IfSubStr{#1}{a}{\node (a) at (0:2ex) {};}{}
							\IfSubStr{#1}{b}{\node (b) at (120:2ex) {};}{}
							\IfSubStr{#1}{c}{\node (c) at (240:2ex) {};}{}
							\IfSubStr{#1}{abc}{ \fill[gray!20] (a.center) -- (b.center) -- (c.center) -- (a.center);}{}
							\IfSubStr{#1}{ab}{\draw (a) -- (b);}{}
							\IfSubStr{#1}{ac}{\draw (a) -- (c);}{}
							\IfSubStr{#1}{bc}{\draw (b) -- (c);}{}
							\IfSubStr{#1}{a}{\node[label=right:$a$] (a) at (0:2ex) {};}{}
							\IfSubStr{#1}{b}{\node[label=left:$b$] (b) at (120:2ex) {};}{}
							\IfSubStr{#1}{c}{\node[label=left:$c$] (c) at (240:2ex) {};}{}
							\draw[gray] (-.6,-.5) rectangle (.6,.5);
						\end{scope}
					}
					]
%					\draw[ystep=.5cm, xstep=.6cm, yshift=0mm] (-.6,-.5) grid ++(1.8,1.5);
					\matrix[
						matrix anchor=south west,
						ampersand replacement=\&
					] (m) {
						\pic {Filtration={ac}}; \& \pic {Filtration={ab ac bc}}; \& \pic {Filtration={abc ac}}; \\
						\pic {Filtration={ac}}; \& \pic {Filtration={ab ac bc}}; \& \pic {Filtration={ab ac bc}}; \\
						\pic {Filtration={a}}; \& \pic {Filtration={ab}};   \& \pic {Filtration={ab}}; \\
					};
				\end{tikzpicture}
			}
		}\hfill
		\def\Setup{
			\Axes(-.95,-.95)(2.5,2.5)
			\Grid[0](1){2}[0](1){2}
		}
		\settowidth\TmpLenA{\scriptsize $C_2(\cup K, K)$}
		\subcaptionbox{\label{fig:filtration:C}}{
			\captionsetup[subfigure]{aboveskip=1pt,belowskip=-1pt}
			\begin{tabular}[b]{ccc}
				\subcaptionbox*{\scriptsize $C_2(K)$}[\TmpLenA]{
					\begin{tikzpicture}[module diagram]
						\Setup
						\FreeModule[blue](2,2){}
					\end{tikzpicture}
				} & 
				\subcaptionbox*{\scriptsize $C_1(K)$}[\TmpLenA]{
					\begin{tikzpicture}[module diagram]
						\Setup
						\FreeModule[blue](1.1,1.1){}
						\FreeModule[blue](0,1){}
						\FreeModule[blue](1,0){}
					\end{tikzpicture}
				} & 
				\subcaptionbox*{\scriptsize $C_0(K)$}[\TmpLenA]{
					\begin{tikzpicture}[module diagram]
						\Setup
						\FreeModule[blue](0,0){}
						\FreeModule[blue](1.1,0.1){}
						\FreeModule[blue](0.1,1.1){}
					\end{tikzpicture}
				}
			\\[3ex]
				\subcaptionbox*{\scriptsize $C_2(\cup K, K)$}[\TmpLenA]{
					\begin{tikzpicture}[module diagram]
						\Setup
						\coordinate (s1) at (1,2);
						\coordinate (s2) at (2,1);
						\coordinate (c) at (0,1);
						\FreeComplement[blue](2,2){}
					\end{tikzpicture}
				} & 
				\subcaptionbox*{\scriptsize $C_1(\cup K, K)$}[\TmpLenA]{
					\begin{tikzpicture}[module diagram]
						\Setup
						\FreeComplement[blue](1.1,1.1){}
						\FreeComplement[blue](1,0){}
						\FreeComplement[blue](0,1){}
					\end{tikzpicture}
				} &
				\subcaptionbox*{\scriptsize $C_0(\cup K, K)$}[\TmpLenA]{
						\begin{tikzpicture}[module diagram]
						\Setup
						\FreeComplement[blue](0,0){}
						\FreeComplement[blue](1.1, 0.1){}
						\FreeComplement[blue](0.1, 1.1){}
					\end{tikzpicture}
				}
			\end{tabular}
		}\hfill
		\subcaptionbox{\label{fig:filtration:H}}{
			\raisebox{7ex}{
				\captionsetup[subfigure]{aboveskip=1pt,belowskip=-1pt}
				\def\Label{\scriptsize $\begin{array}{@{}r@{}l@{}} & \phantom{{}\cong{}} H_1(K) \\ & {}\cong H_2(\cup K, K)\end{array}$}
				\subcaptionbox*{\Label}[1pt+\widthof{\Label}]{
					\begin{tikzpicture}[module diagram]
						\Setup
						\coordinate (b) at (.5,.5);
						\coordinate (d) at (1.5,1.5);
						\scoped[blue]{
							\fill[fill opacity=.2] (b |- inf) |- (b -| inf) -- (d -| inf) -| (d |- inf) -- cycle;
							\draw (b |- inf) |- (b -| inf) (d -| inf) -| (d |- inf);
						}
					\end{tikzpicture}
				}
			}
		}
		\caption{%
			(\subref{fig:filtration:K}) A one-critical $\Z^2$\-/indexed simplicial filtration $K$,
			(\subref{fig:filtration:C}) the associated absolute and relative filtered chain complexes $C_\bullet(K)$ and $C_\bullet(\cup K, K)$,
			(\subref{fig:filtration:H}) the only nonzero absolute/relative (reduced) persistent homology module of $K$.
		}
		\label{fig:filtration}
	\end{figure}
	Let $K$ be the $\Z^2$\-/indexed simplicial filtration shown in \cref{fig:filtration:K}.
	We obtain the absolute and relative filtered chain complexes $C_\bullet(K)$ and $C_\bullet(\cup K, K)$ illustrated in \cref{fig:filtration:C}.
	The absolute and relative homology $H_1(K) \cong H_2(\cup K, K)$ is shown in \cref{fig:filtration:H};
	all other (absolute or relative) homology modules are zero.
	The pictures for the absolute and relative cochain complexes $C^\bullet(K)$ and $C^\bullet(\cup K, K)$
	are obtained by flipping \cref{fig:filtration:C} about the origin.
	By \eqref{eq:uct}, also $H^\bullet(K)$ and $H^\bullet(\cup K, K)$ are obtained by flipping \cref{fig:filtration:H} about the origin.
\end{example}

We remark that for nonempty $K$, the complexes $C^\bullet(K)$ and $C_\bullet(\cup K, K)$ are never complexes of free modules,
and $C^\bullet(\cup K, K)$ is a complex of free modules only if $K$ is a $\Z^1$\-/filtration.

\subsection{Resolutions}
A \emph{free} (or \emph{flat}) \emph{resolution} of a module $M$
is a chain complex $G$ of free (resp.\ flat) modules
concentrated in nonnegative degrees, together with an \emph{augmentation morphism} $\varepsilon\colon G_0 \to M$,
such that the \emph{augmented free \emph{(resp.\ \emph{flat})} resolution} \[\dotsb \to G_1 \to G_0 \stackrel\varepsilon\to M \to 0\] is an exact sequence.
Analogously, an \emph{injective} (or \emph{cofree}) \emph{resolution} of $M$
is a cochain complex $I$ of injective (resp.\ cofree) modules, together with an augmentation morphism $\varepsilon\colon M \to I^0$,
such that the \emph{augmented injective \emph{(resp. \emph{cofree})} resolution} $0 \to M \stackrel\varepsilon\to I^0 \to I^1 \to \dotsb$ is an exact sequence.
It follows from exactness of $(-)^*$ that if $\dotsb \to G_1 \to G_0 \to M$ is an augmented free (or flat) resolution,
then $M^* \to G_0^* \to G_1^* \to \dotsb$ is an augmented cofree (or injective, resp.) resolution.
The \emph{length} $\ell(G)$ of a flat resolution $G$ is the maximal index $d$ such that $G_d\ne 0$.

%\paragraph{Minimality}
%\begin{definition}
	A \emph{(co)homological $d$\-/ball} is a (co)chain complex of (co)free modules of the form
	\[
		\dotsb \to 0 \to F(z) \stackrel\id\to F(z) \to 0 \to \dotsb
		\qquad\text{resp.}\qquad
		\dotsb \to 0 \to I(z) \stackrel\id\to I(z) \to 0 \to \dotsb
	\]
	for some $z\in \Z^N$,
	with the two modules appearing in (co)homological degrees $d$ and $d-1$.
	A (co)chain complex $C$ of finite rank (co)free modules is \emph{trivial} if it is isomorphic to a direct sum of (co)homological balls,
	and \emph{minimal} if it does not contain any (co)free (co)homological ball as a direct summand.
	A (co)free resolution is \emph{minimal} if it is minimal as a (co)chain complex.
	We use the shorthand \MFR\ for \enquote{minimal free resolution}.
%\end{definition}

\begin{example}
	\label{ex:koszul}
	Assume $M$ is a finitely generated $\Z$\-/persistence module with $\barc M=\Set{[b_i,d_i); i = 1,\dotsc,m}$,
	such that for some $0 \leq n\leq m$, we have $d_1,\dotsc,d_n < \infty$ and $d_{n+1},\dotsc,d_m = \infty$.
	Then $\barc M$ gives rise to a minimal free resolution
	\begin{equation}
		\label{eq:barcode-as-resolution}
		0 \to \bigoplus_{i \leq n} F(d_i)
		\xto{\Mat{\Phi}} \bigoplus_{i \leq m} F(b_i)
	\end{equation}
	of $M$, where the graded $(b_1,\dotsc,b_m)\times(d_1,\dotsc,d_n)$\-/matrix $\Mat{\Phi}$ has entries $\Mat{\Phi}_{ij} = \delta_{ij}$.
	This resolution is indeed minimal because there are no intervals with $b_i=d_i$ in $\barc M$.
\end{example}

\begin{example}
	\label{ex:simple-module}
	Let $\SimpleModule$ denote the $\Z^N$\-/persistence module with components $\SimpleModule_0 = \FF$ and $\SimpleModule_z = 0$ for $z \neq 0$.
	It is a simple persistence module, in the sense that it has no proper submodules.
	It has the minimal free resolution
	\begin{equation}
		\label{eq:koszul-complex}
		\mathcal{K}\colon \qquad 0 \to \smashoperator{\bigoplus_{S\in\binom{[N]}{N}}} F(e_S) \xto{\partial_{N}} \dotsb \xto{\partial_1} \smashoperator{\bigoplus_{S\in\binom{[N]}{0}}} F(e_S)
	\end{equation}
	where $e_S \in \Z^N$ denotes the vector with $(e_S)_i = 1$ if $i \in S$ and $(e_S)_i = 0$ otherwise,
	and the boundary maps $\partial_k$ have the nonzero components
	\[
		(\partial_k)_{S, S\setminus \{s_j\}} = (-1)^j\colon F(e_S) \to F(e_{S\setminus\{s_j\}})
	\]
	if $S = \{s_1 \leq \dotsb \leq s_k\}$;
	see \cref{fig:simple-module-resolution}.
	\begin{figure}
		\centering
		\tikzset{
			every module diagram/.append style={baseline=(b)}
		}
		\newcommand{\Setup}{
			\Axes(-1.5,-1.5)(2.5,2.5)
			\Grid[-1](0){2}[-1](0){2}
			\coordinate (g) at (0,0);
			\coordinate (b) at (0,0);
%			\node[blue, generator, "$z$" {above, blue}] at (g) {};
		}
		$\mathcal{K}\colon\quad
		0 \to
		\underset{F\binom{1}{1}}{
			\begin{tikzpicture}[module diagram]
				\Setup
				\FreeModule[blue](1,1){}
			\end{tikzpicture}
		}
		\xto{\Mtx*[r]{1\\-1}}
		\underset{F\binom{1}{0} \oplus F\binom{0}{1}}{
			\begin{tikzpicture}[module diagram]
				\Setup
				\FreeModule[blue](1,0){}
				\FreeModule[blue](0,1){}
			\end{tikzpicture}
		}
		\xto{(1\ 1)}
		\underset{F\binom{0}{0}}{
			\begin{tikzpicture}[module diagram]
				\Setup
				\FreeModule[blue](0,0){}
			\end{tikzpicture}
		}
		\stackrel\varepsilon\to
		\underset{\SimpleModule}{
			\begin{tikzpicture}[module diagram]
				\Setup
				\filldraw[fill opacity=.2, blue] (-.5,-.5) rectangle (.5,.5);
			\end{tikzpicture}
		}
		\to 0$
		\caption{Augmented minimal free resolution of the simple $\Z^2$\-/persistence module $\SimpleModule$.}
		\label{fig:simple-module-resolution}
	\end{figure}
	This complex is also called the \emph{Koszul complex}.
	The augmentation map $\varepsilon\colon F(0) \to \SimpleModule$ is given componentwise by
	$\varepsilon_0 = \id_\FF$ and $\varepsilon_z = 0$ for $z \ne 0$.
\end{example}

\begin{theorem}[{\cite[Corollary~19.7 and Theorem~20.2]{Eisenbud:1995}}]
	\label{thm:minimal-free-resolutions-are-unique}
	Let $M$ be a finitely generated $\Z^N$\-/persistence module.
	Then
	\begin{enumerate}
		\item\label{thm:Hilbert-syzygy} $M$ has a \MFR\ $G$ of length at most $N$ (Hilbert's Syzygy theorem), and
		\item every free resolution of $M$ is isomorphic to the direct sum of $G$ and a trivial complex.
	\end{enumerate}
\end{theorem}

In particular, a \MFR\ of $M$ exists and is unique up to isomorphism of chain complexes, so the following is well defined:

\begin{definition}
	\label{def:betti-numbers}
	The \emph{graded Betti numbers} of a finitely generated $\Z^N$\-/persistence module $M$
	are the graded ranks $\beta_d(M) \coloneqq \rk G_d$, where $G$ is any \MFR\ of $M$.
\end{definition}

\begin{remark}
	The \emph{graded dimension} $\dim M$ of a persistence module $M$ is the multiset of elements of $\Z^N$ that, for each $z\in \Z^N$, has $\dim M_z$ copies of $z$.
	It is well known that the graded Betti numbers satisfy $\beta_d(M) = \dim H_d(M \otimes \mathcal{K})$,
	where $\mathcal{K}$ is the Koszul complex from \cref{eq:koszul-complex}.
\end{remark}

\begin{remark}
	Recall that a graded matrix $\Mat{M}$ is valid if $\rg^{\Mat{M}}_i \leq \cg^{\Mat{M}}_j$ for all $\Mat{M}_{ij} \neq 0$.
	A graded matrix $\Mat{M}$ is \emph{minimal} if $\rg^{\Mat{M}}_i < \cg^{\Mat{M}}_j$ for all $\Mat{M}_{ij} \neq 0$.
	A chain complex of finite rank (co)free modules is minimal if and only if the graded matrices
	representing it (with respect to any generalized bases) are minimal.
\end{remark}

\section{Dualities for multiparameter persistence}
\label{sec:dualities}
The goal of this section is to develop a connection between persistent homology and cohomology in the multiparameter setting.
Specifically, we establish \cref{thm:local-duality,thm:module-resolutions},
which together constitute a multiparameter version of \cref{thm:dSMVJ}.

\subsection{The global dual \texorpdfstring{\Boldmath$(-)^\dagger$}{}}
Recall the pointwise duality functor $(-)^*$ and the $\IHom$-functor from \cref{def:three-functors}.
We define the following global duality functor of persistence modules:
\begin{definition}
	\label{def:dagger}
	Let $\one = (1,\dotsc,1) \in \Z^N$.
	The \emph{global dual} $(-)^\dagger$ is the additive functor
	\[
		(-)^\dagger \coloneqq \IHom(-, F(\one))\colon (\pers{\Z^N})^\op \to \pers{\Z^N}.
	\]
\end{definition}
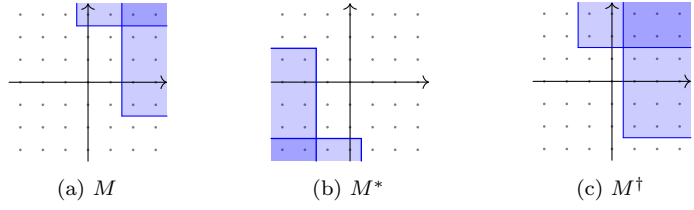
\begin{figure}\centering
	\def\Coords{
		\Axes(-3.5,-3.5)(3.5,3.5)
		\Grid[-3](-2){3}[-3](-2){3}
		\coordinate (a) at (0,3);
		\coordinate (b) at (2,-1);
	}
	\begin{subfigure}{.25\linewidth}\centering
		\begin{tikzpicture}[module diagram]
			\Coords
			\FreeModule[blue](a){}
			\FreeModule[blue](b){}
		\end{tikzpicture}
		\caption{$M$}
	\end{subfigure}
	\begin{subfigure}{.25\linewidth}\centering
		\begin{tikzpicture}[module diagram]
			\Coords
			\InjModule[blue]($(0,0)-(a)$){}
			\InjModule[blue]($(0,0)-(b)$){}
		\end{tikzpicture}
		\caption{$M^*$}
	\end{subfigure}
	\begin{subfigure}{.25\linewidth}\centering
		\begin{tikzpicture}[module diagram]
			\Coords
			\FreeModule[blue]($(1,1)-(a)$){}
			\FreeModule[blue]($(1,1)-(b)$){}
		\end{tikzpicture}
		\caption{$M^\dagger$}
	\end{subfigure}
	\caption{A free module $M$, its pointwise dual $M^*$, and its global dual $M^\dagger$.}
	\label{fig:duals}
\end{figure}
Including the shift by $\one$ in the definition of $(-)^\dagger$ allows us to state our main results without any shifts.
Like $(-)^*$, the functor $(-)^\dagger$ is contravariant and thus maps chain complexes to cochain complexes.
Unlike $(-)^*$, the functor $(-)^\dagger$ is not exact.
It preserves finite rank free modules, because
\begin{equation}
	\label{eq:dagger-acting-on-free}
	F(z)^\dagger \cong F(\one-z).
\end{equation}
%Therefore, if $M$ is a finite rank free module, then so is $M^\dagger$.
This in particular implies that if $K$ is a finite one-critical simplicial $\Z^N$\-/filtration, then
\begin{equation}
	C_d(K)^\dagger = \bigl(\smash[b]{\smashoperator{\toplus_{\sigma \in \cup K^{(d)}}}} F(g(\sigma))\bigr)^\dagger = \smash[b]{\smashoperator{\bigoplus_{\sigma \in \cup K^{(d)}}}} F(\one-g(\sigma))
\end{equation}
is a free module.

\begin{proposition}
	\label{rmk:dagger-in-1D}
	For a finite $\Z$\-/indexed filtration $K$, the cochain complexes $C_\bullet(K)^\dagger$ and $C^\bullet(\cup K, K)$ are naturally isomorphic (where naturality is with respect to $K$).
\end{proposition}
\begin{proof}
	By definition, an element $\kappa \in (C_d(K)^\dagger)_z$
	is a morphism $\kappa\colon C_d(K) \to F(1-z)$; note that $\kappa_w(C_d(K)_w) = 0$ whenever $w \leq -z$.
	By definition of $C^\bullet(\cup K, K)$ and exactness of $(-)^*$, the cochain complex $C^\bullet(\cup K, K)$ has components
	\begin{align*}
		C^d(\cup K, K)_z & \stackrel{\mathclap{\text{def}}}{=} \coker\bigl(C_d(K_{-z}) \into C_d(\cup K)\bigr)^* \\
		                 & \cong \ker\bigl(C_d(\cup K)^* \onto C_d(K_{-z})^*\bigr)                                   \\
		                 & = \Set[\big]{\lambda \in C_d(\cup K)^*; \lambda\bigl(C_d(K_{-z})\bigr) = 0};
	\end{align*}
	in other words, an element of $C^d(\cup K, K)_z$ is a linear map $C_d(\cup K)\to \FF$ that sends $C_d(K_{-z})$ to zero.
    Therefore, we obtain a well defined morphism
    \[
    	\phi^d_z \colon (C_d(K)^\dagger)_z \longrightarrow C^d(\cup K, K)_z
    \]\vadjust pre {\vskip-\belowdisplayskip}%
    with
    \[
    	\phi^d_z(\kappa)\colon C_d(\cup K) = \colim_w C_d(K_w) \xto{\colim_w \kappa_w} \colim_w F(1-z)_w = \FF
    \]
    for $\kappa \in (C_d(K)^\dagger)_z$.
    It is easily verified that $\phi^d_z$ has an inverse map
    \[
		\psi^d_z\colon C^d(\cup K, K)_z \longrightarrow (C_d(K)^\dagger)_z
	\]
	with
	\[
		\psi^d_z(\lambda)_w\colon C_d(K_w) \into C_d(\cup K) \xto{\lambda} \FF \onto F(1-z)_w,
	\]
	and that $\phi^d_z$ and $\psi^d_z$ are natural in $z$ and $K$ and commute with the respective coboundary maps.
	Therefore we obtain a pair of mutually inverse natural isomorphisms of chain complexes, natural in $K$:
	 \[
	 	\phi \colon  C_\bullet(K)^\dagger \rightleftarrows C^\bullet(\cup K, K)\noloc \psi.\qedhere
	 \]
\end{proof}

See also \cref{rmk:relative-cochains} for a more conceptual approach to the \lcnamecref{rmk:dagger-in-1D}.

\subsection{Dualities and matrices}

\begin{definition}
	Let $M$ be a finite rank flat module and $b\colon \bigoplus_{i=1}^n F(z_i) \to M$ be a generalized basis of $M$, with $z_i \in \lZ^N$.
	The generalized basis of $M^*$ \emph{dual} to $b$ is given by the induced isomorphism
	\[
		\bigoplus_{i=1}^n I(-z_i) \cong
		\bigoplus_{i=1}^n F(z_i)^* \xto{(b^{-1})^*}
		M^*.
	\]
	The generalized basis of $M^\dagger$ \emph{dual} to $b$ is
	given by the induced isomorphism
	\[
		\bigoplus_{i=1}^n F(\one-z_i) \cong
		\bigoplus_{i=1}^n F(z_i)^{\dag} \xto{(b^{-1})^\dagger}
		M^{\dag}.
	\]
\end{definition}

The \emph{graded transpose} of a graded $\bm{r} \times \bm{c}$\-/matrix $\Mat{M}$
is the graded $(-\bm{c})\times (-\bm{r})$\-/matrix $\Mat{M}^T$
with $\uMat{\Mat{M}^T} = \uMat{\Mat{M}}^T$.
For a graded matrix $\Mat{M}$, let $\Mat{M}\Shift{z}$ be the graded matrix with entries $\uMat{\Mat{M}\Shift{z}} = \uMat{\Mat{M}}$,
row grades $\rg^{\Mat{M}\smash{\Shift{z}}}_i = \rg^{\Mat{M}}_i - z$
and column grades $\cg^{\Mat{M}\smash{\Shift{z}}}_j = \cg^{\Mat{M}}_j - z$ for all $i$, $j$.
The graded matrix $\Mat{M}$ is valid if and only if $\Mat{M}^T$ is valid if and only if $\Mat{M}\Shift{z}$ is valid.
We omit the straightforward proof of the following statement; see also \cite[Lemmas~3.18, 3.25]{Miller:2000}:

\begin{lemma}
	\label{thm:dual-matrices}
	If a valid graded matrix $\Mat{\Phi}$ represents a morphism $\phi\colon F \to G$ of free modules, then
	with respect to the dual generalized bases,
	\settowidth\MBoxWidth{$\Mat{\Phi}^T\Shift{\one}$}%
	\begin{enumerate}
		\item $\MBox{\Mat{\Phi}^T}$ represents the morphism $\phi^*\colon G^* \to F^*$ of cofree modules, and
		\item $\Mat{\Phi}^T\Shift{\one}$ represents the morphism $\phi^\dagger\colon G^\dagger \to F^\dagger$ of free modules.
	\end{enumerate}
\end{lemma}

\subsection{Main results}
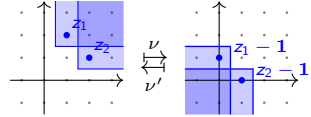
\begin{figure}[b!]
	\centering
	\(
		\begin{tikzpicture}[module diagram, baseline=(b)]
			\Axes(-1.5,-1.5)(3.5,3.5)
			\Grid[-1](0){3}[-1](0){3}
			\coordinate (b) at (0,.5);
			\FreeModule[blue](1,2)[generator, label=above right:$z_1$]{}
			\FreeModule[blue](2,1)[generator, label=above right:$z_2$]{}
		\end{tikzpicture}
		\ \overset{\nu}{\underset{\nu'}{\mapstofrom}}\
		\begin{tikzpicture}[module diagram, baseline=(b)]
			\Axes(-1.5,-1.5)(3.5,3.5)
			\Grid[-1](0){3}[-1](0){3}
			\coordinate (b) at (0,.5);
			\InjModule[blue](0,1)[generator, label={[label distance=4pt]10:$z_1-\one$}]{}
			\InjModule[blue](1,0)[generator, label={[label distance=4pt]10:$z_2-\one$}]{}
		\end{tikzpicture}
	\)
	\caption{A free module $F = F(z_1) \oplus F(z_2)$ (left)
	and its image $\nu F = I(z_1-\one) \oplus I(z_2-\one)$ under the Nakayama functor (right).}
	\label{fig:Nakayama}
\end{figure}
The \emph{Nakayama functor} is the additive covariant functor $\nu \coloneqq ((-)^\dagger)^*\colon \pers{\Z^N} \to \pers{\Z^N}$.
When restricted to finite rank free modules, $\nu$ is invertible with inverse $\nu' = ((-)^*)^\dagger$,
where $\nu\colon F(z) \mapstofrom I(z-\one)\noloc \nu'$; see \cref{fig:Nakayama}.
In this section, we establish \cref{thm:local-duality,thm:module-resolutions} by virtue of \cref{thm:main} below,
which provides a natural isomorphism between the homology of $C$ and $\nu C$.

For $C$ a chain complex and $i \in \Z$, the \emph{(homologically) shifted complex} $C[i]$
is the chain complex with $C[i]_d = C_{i+d}$.
Similarly, for $C$ a cochain complex and $i \in \Z$, we define $C[i]$ to be the cochain complex with $C[i]^d = C^{i+d}$.
The homological shift $[-]$ should not be confused with the graded shift $\Shift{-}$;
the former applies to (co)chain complexes of persistence modules and shifts the (co)homological degree,
while the latter applies to persistence modules and shifts the $\Z^N$\-/grade.
We use the notation $\IntSet{N} \coloneqq \{1,\dotsc,N\}$.

\begin{definition}
	\label{def:colim}
	For $S = \{s_1,\dotsc,s_k\} \subseteq \IntSet{N}$,
	let $q_S\colon \Z^N \to \Z^{N-\abs{S}}$ be the poset map that forgets the components indexed by $S$.
	We define the additive functors
	\[
		\begin{aligned}
			\colim_S \colon \pers{\Z^N}          & \to \pers{\Z^{N-\abs{S}}}, & (\colim_S M)_z & = \colim_{w \in q_S^{-1}(z)} M_w, \\
			\Delta_S\colon \pers{\Z^{N-\abs{S}}} & \to \pers{\Z^N},           & (\Delta_S L)_w & = L_{q_S(w)}                      \\
			\lim_S \colon \pers{\Z^N}            & \to \pers{\Z^{N-\abs{S}}}, & (\lim_S M)_z   & = \lim_{w \in q_S^{-1}(z)} M_w,
		\end{aligned}
	\]
	where the structure maps of $\colim_S M$ and $\lim_S M$ are obtained by the functoriality of (co)limits.
	Thus, $\lim_S M$ is the $(N - \abs{S})$\-/parameter persistence module obtained by taking the limit of $M$ with respect to the parameters indexed by $S$,
	and similarly for $\colim_S M$,
	while $\Delta_S L$ is the $N$\-/parameter persistence module obtained by copying $L$ along the parameters indexed by $S$.
	Further, we let $\Lim_S \coloneqq \Delta_S \lim_S$ and $\Colim_S \coloneqq \Delta_S \colim_S$.
\end{definition}
%For example, if $N = 3$ and $S = \{1, 3\}$, then $\colim_{\{1,3\}} M$ is the $\Z$\-/persistence module with
%$(\colim_{\{1, 3\}} M)_{z} = \colim_{(z_1, z_3) \in \Z^2} M_{(z_1, z, z_3)}$.

The module $\Delta_S M$ is constant along the parameters specified by $S$.
Thus, the same is true for $\Colim_S M$ and $\Lim_S M$.
The functor $\Delta_S$ is exact, right adjoint to $\colim_S$, and left adjoint to $\lim_S$.
The functor $\colim_S$ (and thus also $\Colim_S$) is also exact, because it is constructed pointwise via a colimit over a directed system of vector spaces,
and colimits over directed systems are exact \cite[Theorem~2.6.15]{Weibel:2003}.
In particular, $\colim_S$, $\Delta_S$, and thus $\Colim_S$ commute with (co)homology of (co)chain complexes.

\begin{definition}
	\label{def:eventually-acyclic}
	A  $\Z^N$\-/persistence module $M$ \emph{eventually vanishes} if the \emph{support} $\supp M \coloneqq \Set{z \in \Z^N ; M_z \neq 0}$ of $M$ is bounded above.
	A chain complex $C$ in $\pers{\Z^N}$ is called \emph{eventually acyclic} if $H_d(C)$ eventually vanishes for all $d$.
	A $\Z^N$\-/filtered simplicial  complex $K$ is \emph{eventually acyclic} if $C_\bullet(K)$ is.
\end{definition}
\begin{remark}
	A finitely generated persistence module $M$ eventually vanishes if and only if $\colim_S M = 0$ for all nonempty $S \subseteq \IntSet{N}$
	(equivalently: for all $S$ with $\abs{S}=1$),
	hence the phrasing \enquote{eventually vanishes}.
	In the language of commutative algebra, a finitely generated persistence module $M$, viewed as a graded $\FF[x_1,\dotsc,x_N]$\-/module,
	eventually vanishes if and only if $M$ is $\m$\-/torsion,  
	where $\m = (x_1,\dotsc,x_N)$.
	Equivalently, $M$ has \emph{finite length};
	i.e., every strictly descending chain $M \supsetneq M_1 \supsetneq M_2 \supsetneq \dotsb$ of submodules is finite.
\end{remark}

\phantomsection\label{sec:derived-cat}
Recall that the derived category $\D(\pers{\Z^N})$ is the localization of the category $\Ch(\pers{\Z^N})$ of chain complexes of $\Z^N$\-/modules with respect to quasi-isomorphisms \cite[\S10.4]{Weibel:2003}.
Thus, the objects are chain complexes of $\Z^N$\-/persistence modules,
and morphisms are obtained by formally inverting the quasi-isomorphisms.
In fact, any morphism $C \to D$ in $\D(\pers{\Z^N})$ can be represented by a zig-zag $C \xleftarrow{\sim} Z \to D$ of morphisms in $\Ch(\pers{\Z^N})$ for some $Z$, where $Z \xto{\sim} C$ is a quasi-isomorphism.

\begin{theorem}
	\label{thm:main}
	Let $C$ be a chain complex of finite rank free $\Z^N$\-/persistence modules.
	There exists a natural morphism $\varphi\colon \nu C[N] \to C$ in the derived category $\mathcal{D}(\pers{\Z^N})$ such that
	\begin{enumerate}
		\item\label{thm:main:1} if each module $H_d(C),\dotsc,H_{d+N}(C)$ eventually vanishes, then
		\[
			H_d(\varphi)\colon H_{d+N}(\nu C) \to H_d(C)
		\]
		is an isomorphism.
		\item\label{thm:main:2} If $C$ is eventually acyclic, then $\varphi$ is an isomorphism in $\mathcal{D}(\pers{\Z^N})$.
	\end{enumerate}
\end{theorem}

%\begin{theorem}
%	\label{thm:main}
%	If $C$ is an eventually acyclic chain complex of finite rank free $\Z^N$\-/persistence modules, then $C$ and $\nu C[N]$ are naturally isomorphic in the derived category $\D(\pers{\Z^N})$.
%	In particular, there is a natural isomorphism $H_d(C) \cong H_{d+N}(\nu C)$ for all $d$.
%\end{theorem}

\begin{proof}
	\label{proof:main}
	For the construction of $\varphi$, define for each $0 \leq i \leq N$ the exact functor
	\begin{align*}
		\Omega_i\colon \pers{\Z^N} & \longrightarrow \pers{\Z^N},                                                         \\
		M                          & \longmapsto \smashoperator{\bigoplus_{S \in \binom{\IntSet{N}}{N-i}}} \Colim_S M,
	\end{align*}
	where $\binom{\IntSet{N}}{i} \coloneqq \Set{ S \subseteq \IntSet{N}; \abs{S} = i }$.
	For every $S' \subseteq S$ and $M \in \pers{\Z^N}$, there is a natural morphism $c^M_{S,S'}\colon \Colim_{S'} M \to \Colim_{S} M$.
	These give rise to natural morphisms
	\begin{equation*}
		\omega_i^M\colon \Omega_i M = \smashoperator{\bigoplus_{S' \in \binom{\IntSet{N}}{N-i}}} \Colim_{S'} M \longrightarrow \smashoperator{\bigoplus_{S \in \binom{\IntSet{N}}{N-i+1}}} \Colim_S M = \Omega_{i-1} M,
	\end{equation*}
	defined by their nonzero components
	\begin{equation*}
		(-1)^j c^M_{S, S \setminus \{s_j\} }\colon \Colim_{S \setminus \{s_j\} } M \to \Colim_{S} M
	\end{equation*}
	for $S = \{s_1 < \dotsb < s_j < \dotsb < s_{N-i}\}$.
	The sign rule ensures that the maps $\omega_i^M$ form a chain complex
	\begin{equation}
		\label{proof:main:les-1}
		\let\to\longrightarrow
		\Omega M\colon\quad
		0 \to \underbrace{\Omega_N M}_M \xto{\omega_N^M} \Omega_{N-1} M \xto{\omega_{N-1}^M} \dotsb \xto{\omega_1^M} \Omega_0 M \to 0.
	\end{equation}
	In fact, $\Omega$ is an additive functor from $\pers{\Z^N}$ to $\Ch(\pers{\Z^N})$.
	If $M$ is free, then $\Omega M$ is a flat resolution of $\nu M$,
	whose augmentation morphism $\varepsilon^{M}$ is the canonical morphism
	\begin{equation*}
		\Omega_0 M = \Colim_{\IntSet{N}} M \cong \Delta_{\IntSet{N}} \FF^{\abs{\rk M}} \cong \Lim_{\IntSet{N}} \nu M \longrightarrow \nu M;
	\end{equation*}
	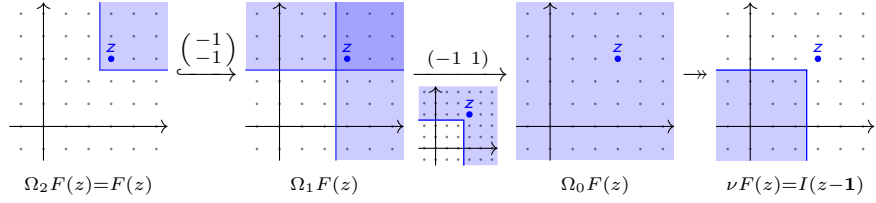
\begin{figure}
		\centering
		\tikzset{
			every module diagram/.append style={baseline=(b)}
		}
		\newcommand{\Setup}{
			\Axes(-1.5,-1.5)(5.5,5.5)
			\Grid[-1](0){5}[-1](0){5}
			\coordinate (g) at (3,3);
			\coordinate (b) at (0,2);
			\node[blue, generator, "$z$" {above, blue}] at (g) {};
			\path[use as bounding box] (-1.5,-2) -- (inf);
		}
		$
		\underset{\Omega_2 F(z) = F(z)}{
			\begin{tikzpicture}[module diagram]
				\Setup
				\FreeModule[blue](g){}
			\end{tikzpicture}
		}
		\xinto{\!\!\Mtx{-1\\-1}\!\!}
		\underset{\Omega_1 F(z)}{
				\begin{tikzpicture}[module diagram]
				\Setup
				\FreeModule<x>[blue](g){}
				\FreeModule<y>[blue](g){}
			\end{tikzpicture}
		}
		\xrightarrow[{
			\begin{tikzpicture}[module diagram,x=1.5mm,y=1.5mm]
				\Setup
				\InjComplement[blue](g){}
			\end{tikzpicture}
		}]{(-1\ 1)}
		\underset{\Omega_0 F(z)}{
			\begin{tikzpicture}[module diagram]
				\Setup
				\FreeModule<xy>[blue](g){}
			\end{tikzpicture}
		}
		\onto
		\underset{\nu F(z) = I(z-\one)}{
			\begin{tikzpicture}[module diagram]
				\Setup
				\InjModule[blue]($(g)-(1,1)$){}
			\end{tikzpicture}
		}$
		\caption{The augmented flat resolution $\Omega F(z)$ of $\nu F(z) = I(z-\one)$ for $N = 2$.
			The module below the arrow denotes the image of that morphism.}
		\label{fig:Koszul-complex:normal}
	\end{figure}%
	see \cref{fig:Koszul-complex:normal} for an illustration.
	Denote by $\mathcal{F}(\pers{\Z^N})$ the full subcategory of $\pers{\Z^N}$ consisting of the free modules.
	For $0 \leq i \leq N$, define the functors
	\begin{equation}
		\label{eq:def-U}
		U^{(i)}\colon \mathcal{F}(\pers{\Z^N}) \to \pers{\Z^N},\quad
		M \mapsto \im \omega^M_i.
	\end{equation}
	Note that $U^{(i)}$ is defined only up to unique natural isomorphism.
	The objects $U^{(i)} M$ form the following functorial unsplicing of the augmented flat resolution $\Omega M \to \nu M$
	into short exact sequences:
	\begin{equation}
		\label{proof:main:les-unspliced}
		\begin{tikzcd}[row sep={0.9cm,between origins}, column sep={0.95cm,between origins}, baseline=(C.base), trim left=(B), trim right=(A)]
			               &                                     &                  & 0 \ar[dr]                      &                            & 0                               &                     &   & [-1.15cm]      & [-1.4cm]  0 \ar[dr] &                  &                           &                         &                            &                                 & 0                        & [0.1cm] \\
			               &                                     &                  &                                & U^{(N-1)} M\ar[dr] \ar[ur] &                                 &                     &   &                &                     & U^{(2)} M\ar[dr] &                           &                         &                            & U^{(0)} M\ar[dr, equal] \ar[ur] &                          &         \\
			0 \ar[r]       & |[alias=B]|M \ar[rr] \ar[dr, equal] &                  & \Omega_{N-1} M \ar[rr] \ar[ur] &                            & \Omega_{N-2} M \ar[rrr] \ar[dr] &                     &   & \cdots\ar[rrr] &                     &                  & \Omega_1 M \ar[rr]\ar[dr] &                         & \Omega_0 M \ar[rr] \ar[ur] &                                 & |[alias=A]| \nu M \ar[r] & 0.      \\
			               &                                     & U^{(N)} M\ar[ur] &                                &                            &                                 & U^{(N-2)} M \ar[dr] &   &                &                     &                  &                           & U^{(1)}M \ar[dr]\ar[ur] &                            &                                 &                          &         \\
			|[alias=C]| {} & 0  \ar[ur]                          &                  &                                &                            &                                 &                     & 0 &                &                     &                  & 0 \ar[ur]                 &                         & 0                          &                                 &                          &
		\end{tikzcd}
	\end{equation}
	Applying this degreewise to the chain complex $C$ of free modules
	yields a long exact sequence $\Omega C$ of chain complexes of flat modules,
	which unsplices into short exact sequences
	\[
		\let\to\longrightarrow
		0 \to U^{(i+1)} C \to \Omega_i C \to U^{(i)} C \to 0
	\]
	in $\Ch(\pers{\Z^N})$ for $0 \leq i < N$.
	% The chain map $\cone(U^{(i+1)} C \to \Omega_i C) \to U^{(i)} C$ given by compsing the projection of the mapping cone onto $\Omega_i C$ with the map $\Omega_i C \to U^{(i)}$ is a quasi-isomorphism.
	Each of these short exact sequences extends to an exact triangle
	\begin{equation}
		\label{proof:main:connect}
		\let\to\longrightarrow
		U^{(i+1)} C \to \Omega_i C \to U^{(i)} C \stackrel{\partial^{(i)}}{\to} U^{(i+1)}C[-1]
	\end{equation}
	in the derived category $\D(\pers{\Z^N})$%
	%, where the connecting morphism $\partial^{(i)}$ is
	% obtained by inverting the above quasi-isomorphism with the projection
	% of the mapping cone onto $U^{(i+1)} C[-1]$
	; see \cite[\S 10.4.9]{Weibel:2003} for details.
	Composing suitable shifts of the morphisms $\partial^{(i)}$ yields the desired morphism
	\[
		\varphi\colon \quad \nu C[N] = U^{(0)}C[N] \xto{\partial^{(0)}[N]} \dotsb \xto{\partial^{(N-1)}[1]} U^{(N)}C = C.
	\]
	Naturality of this morphism follows from the fact that all steps in the construction are functorial.

	We now show that $\varphi$ has the claimed properties.
	For \proofref{thm:main:1}, note that the triangles \eqref{proof:main:connect} give rise to the long exact sequences
	\begin{multline}
		\label{proof:main:les}
		\dotsb \to H_{d+N-i}(\Omega_i C) \to H_{d+N-i}(U^{(i)}C) \newlinearrow[1ex]{}[.2ex]{}
		H_{d+N-i-1}(U^{(i+1)}C) \to H_{d+N-i-1}(\Omega_{i+1} C) \to \dotsb
	\end{multline}
	in $\pers{\Z^N}$ for $0 \leq i < N$; see \cite[Corollary~10.1.4]{Weibel:2003}.
	Recall that $\Colim_S$ is an exact functor and thus commutes with homology.
	This implies $H_d(\Omega_{i} C) \cong \smash[b]{\bigoplus_{S \in \binom{\IntSet{N}}{N-i}}} \Colim_S H_d(C)$.
	From the assumption, we get
	\begin{equation*}
		H_{d+N-i}(\Omega_{i} C) = H_{d+N-i-1}(\Omega_{i+1} C) = 0
	\end{equation*}
	for $0 \leq i < N$.
	Therefore, the long exact sequences \cref{proof:main:les}
	show that for each $i$, the connecting homomorphism $\partial^{(i)}$ from \eqref{proof:main:connect}
	induces an isomorphism
	\[
		H_{d+N-i}(U^{(i)}C) \xto{\cong} H_{d+N-i-1}(U^{(i+1)}C).
	\]
	Thus, $\varphi$ induces an isomorphism
	\[
		H_{d+N}(\nu C) \xto{\cong} H_d(C).
	\]

	For \proofref{thm:main:2}, recall that a morphism in $\mathcal{D}(\pers{\Z^N})$ is an isomorphism
	if it induces isomorphisms in homology in all degrees.
	If $C$ is eventually acyclic, then by definition, $H_q(C)$ eventually vanishes for all $q$.
	Therefore, \proofref{thm:main:2} follows directly from \proofref{thm:main:1}.
\end{proof}

Identifying cochain complexes with chain complexes by negating homological degrees,
one directly obtains the following dual version of \cref{thm:main}:

\begin{theoremmbis}
	\label{thm:main:dual}
	Let $C$ be a cochain complex of finite rank free $\Z^N$\-/persistence modules.
	There exists a natural morphism $\varphi\colon \nu C \to C[N]$ in the derived category $\mathcal{D}(\pers{\Z^N})$ such that
	\begin{enumerate}
		\item\label{thm:main:dual:1} if each module $H^{d}(C),\dotsc,H^{d+N}(C)$ eventually vanishes, then
		\[
			H^d(\varphi)\colon H^d(\nu C) \to H^{d+N}(C)
		\]
		is an isomorphism.
		\item\label{thm:main:dual:2} If $C$ is eventually acyclic, then $\varphi$ is an isomorphism in $\mathcal{D}(\pers{\Z^N})$.
	\end{enumerate}
\end{theoremmbis}

\begin{remark}
	A careful scrutiny of the proof of \cref{thm:main} reveals that \cref{thm:main:1,thm:main:dual:1}
	hold under weaker hypotheses, as follows:
	\begin{enumerate}
		\item  In \cref{thm:main:1}, for $H_d(\varphi)\colon H_{d+N}(\nu C) \to H_d(C)$ to be an isomorphism,
		it suffices to assume only that
		\[
			\colim_S H_{d+i}(C) = \colim_S H_{d+i-1}(C) = 0
		\]
		for all $i = 1,\dotsc,N$ and all $S \in \binom{\IntSet{N}}{i}$.
		\item  In \cref{thm:main:dual:1}, for $H^d(\varphi)\colon H^d(\nu C) \to H^{d+N}(C)$ to be an isomorphism,
		it suffices to assume only that
		\[
			\colim_S H^{d+i}(C) = \colim_S H^{d+i+1}(C) = 0
		\]
		for all $i = 0,\dotsc,N-1$ and all $S \in \binom{\IntSet{N}}{N-i}$.
	\end{enumerate}
\end{remark}

We are now ready to prove \cref{thm:local-duality,thm:module-resolutions}:

\begin{proof}[Proof of \cref{thm:local-duality}]
	\phantomsection\label{proof:local-duality}
	Let $C$ be a chain complex of free persistence modules, such that each module $H_d(C),\dotsc,H_{d+N}(C)$ eventually vanishes.
	With $C^\dagger = (\nu C)^*$, \cref{eq:uct,thm:main:1} give natural isomorphisms
	\begin{equation*}
		H^{d+N}(C^\dagger)
		= H^{d+N}((\nu C)^*)
		\cong H_{d+N}(\nu C)^*
		\cong H_d(C)^*.\qedhere
	\end{equation*}
\end{proof}

\begin{proof}[Proof of \cref{thm:module-resolutions}]
	\phantomsection\label{proof:module-resolutions}
	Let $M$ be a finitely generated, eventually vanishing persistence module,
	and let $G$ be a free resolution of $M$ of length $\ell\leq \infty$.
	Then $G^\dagger[N]$ is a cochain complex of free modules concentrated in degrees $-N,\dotsc,\ell(G)-N$.
	\Cref{thm:local-duality} implies
	\[
		H^d(G^\dagger[N]) \cong H^{d+N}(G^\dagger) \cong H_d(G)^* \cong
		\begin{cases}
			M^* & \text{if $d = 0$}, \\
			0   & \text{otherwise}.
		\end{cases}
	\]
	This shows that $(G^\dagger[N])_0 = (G_N)^\dagger \neq 0$, so $\ell(G) \geq N$.
	If $G$ is minimal, then Hilbert's syzygy theorem (\cref{thm:Hilbert-syzygy}) implies that $\ell(G)=N$,
	proving \proofref{thm:module-resolutions:at-least-N}.
	
	For \proofref{thm:module-resolutions:dual}, assume that $\ell(G) = N$.
	Then $G^\dagger[N]$ is concentrated in (cohomological) degrees $-N,\dotsc,0$,
	so there is a surjection $\varepsilon\colon (G^\dagger[N])^0 \onto H^0(G^\dagger[N]) \cong M^*$.
	We see that when viewed as a chain complex concentrated in (homological) degrees $0,\dotsc,N$,
	the complex $\nu G^\dagger[N]$ is a free resolution of $M^*$, with augmentation map $\varepsilon$.
	Minimality is preserved since $(-)^*$ and $(-)^\dagger$ and thus also $\nu$ commute with finite direct sums and preserve (co)homological balls.
\end{proof}

Recall from \cref{rmk:dagger-in-1D} that in the one-parameter case, for $K$ a finite $\Z$\-/indexed filtration, we have a natural isomorphism $C_\bullet(K)^\dagger \cong C^\bullet(\cup K, K)$.
Thus, when $H_d(\cup K) = H_{d+1}(\cup K) = 0$, \cref{thm:local-duality} recovers the isomorphism $H_d(K)^* \cong H^{d+1}(\cup K, K)$ of \cref{thm:dSMVJ}.

\begin{remark}[Relative cochains]
	\label{rmk:relative-cochains}
	Let $C \coloneqq C_\bullet(K)^\dagger$ for a finite simplicial $\Z^N$\-/filtration $K$,
	and recall the functors $U^{(i)}\colon \mathcal{F}(\pers{\Z^N}) \to \pers{\Z^N}$ from \cref{eq:def-U}.
	Observing that $\Omega_0 C \cong C^\bullet(\cup K)$ and $\nu C \cong C^\bullet(K)$,
	one can show that $U^{(1)} C \cong C^\bullet(\cup K, K)$.
	In particular, for the case $N=1$, this gives back the natural isomorphism
	$C_\bullet(K)^\dagger \cong C^\bullet(\cup K, K)$ from \cref{rmk:dagger-in-1D}.
\end{remark}

\subsection{Examples}
The following example illustrates \cref{thm:local-duality} in the two-parameter case.

\begin{example}
  For $N=2$, let $C$ be the chain complex given by
	\begin{align*}
		C_{-1} & = F\tbinom{0}{0}  &        C_0 & = F\tbinom{0}{1} \oplus F\tbinom{1}{0} & C_d        & = F\tbinom{d}{d+1} \oplus F\tbinom{d+1}{d},\\
		       &                   & \partial_0 & = (1\:1)                               & \partial_d & = \Mtx{(-1)^d & 1 \\ 1 & (-1)^d}            
	\end{align*}
	for $d > 0$ and $C_d = 0$ for $d < -1$.
	Then
	\begin{align*}
		H_d(C)         & = \begin{cases} \SimpleModule\Shift{-d-1} & \textup{for $d \geq -1$}, \\  0 & \textup{for $d < -1$},\end{cases} &
		H^d(C^\dagger) & =\begin{cases} \SimpleModule\Shift{d-1}  & \textup{for $d \geq 1$}, \\ 0 & \textup{for $d < 1$},\end{cases}
	\end{align*}
		where $\SimpleModule \in \pers{\Z^N}$ is  the simple module concentrated at degree $0$ from \cref{ex:simple-module}.
		Therefore,
	\[
		H^{d+2}(C^\dagger)^* = \SimpleModule\Shift{d+1}^* = \SimpleModule\Shift{-d-1} \cong H_d(C)
	\] for all $d$,
	as claimed by \cref{thm:local-duality}.
\end{example}

We next give a matrix formulation of \cref{thm:module-resolutions}, along with two examples illustrating it.

\begin{proposition}
	\label{thm:main-for-matrices}
	Let $M$ be a finitely generated, eventually vanishing $\Z^N$\-/persistence module.
	For valid graded matrices $\Mat{\Phi}_1,\dotsc,\Mat{\Phi}_N$, the following are equivalent:
	\settowidth{\dimen24}{$\scriptstyle \Mat{\Phi}_1^T\Shift{-\one}$}
	\settowidth{\dimen22}{$M^*$}
	\settowidth{\dimen20}{$G_N^\dagger$}
	\begin{enumerate}
		\item $0 \to \mathmakebox[\dimen20][c]{G_N}         \xto{\mathmakebox[\dimen24][c]{\Mat{\Phi}_N}}                 \dotsb \xto{\mathmakebox[\dimen24][c]{\Mat{\Phi}_1}}                 \mathmakebox[\dimen20][c]{G_0}        $ is a free resolution of $M$,
		\item $0 \to \mathmakebox[\dimen20][c]{G_0^\dagger} \xto{\mathmakebox[\dimen24][c]{\Mat{\Phi}_1^T\Shift{-\one}}}  \dotsb \xto{\mathmakebox[\dimen24][c]{\Mat{\Phi}_N^T\Shift{-\one}}}  \mathmakebox[\dimen20][c]{G_N^\dagger}$ is a free resolution of $M^*$.
	\end{enumerate}
\end{proposition}
\begin{proof}
	This follows from \cref{thm:dual-matrices,thm:module-resolutions}.
\end{proof}

\begin{example}
	\label{ex:1D-dual-resolution}
	Let $M$ be a finitely generated, eventually vanishing $\Z$\-/persistence module.
	Recall that $M$ has a barcode decomposition $M \cong \bigoplus_{i=1}^n \FF_{[b_i, d_i)}$ with $b_i$, $d_i \in \Z$ for all $i$.
	With $\bm{b} = (b_1,\dotsc,b_n)$ and $\bm{d} = (d_1,\dotsc,d_n)$,
	the graded $\bm{b}\times\bm{d}$\-/matrix $\Mat{\Phi}$ with $\Mat{\Phi}_{ij} = \delta_{ij}$
	is valid and represents the following \MFR\ of $M$:
	\[
		0 \to \bigoplus_{i=1}^n F(d_i) \xto{\Mat{\Phi}} \bigoplus_{i=1}^n F(b_i).
	\]
	The pointwise dual module $M^*$ has the barcode decomposition
	\[
		M^* \cong \bigoplus_{i=1}^n \FF_{(-d_i, -b_i]} = \bigoplus_{i=1}^n \FF_{[1-d_i, 1-b_i)},
	\]
	where the equality simply swaps open and closed ends of the intervals.
	This decomposition corresponds to the following \MFR\ of $M^*$:
	\[
		0 \to \bigoplus_{i=1}^n F(1-b_i) \xto{\Mat{\Phi}^T\Shift{-1}} \bigoplus_{i=1}^n F(1-d_i).
	\]
\end{example}

\begin{example}
	\label{example:converting-free-and-injective-resolutions}
	\tikzset{
		every module diagram/.append style={
			x=2mm, y=2mm,
			baseline={([yshift=-1ex]current bounding box.center)}
		}
	}
	\tikzcdset{diagrams={ampersand replacement=\&, column sep={20ex,between origins}}}
	\newcommand{\coordinatesA}{
		\coordinate (G1)  at=(0,2);
		\coordinate (G2)  at=(2,0);
		\coordinate (R1)  at=(0,4);
		\coordinate (R2)  at=(4,0);
		\coordinate (R3)  at=(2,2);
		\coordinate (S1)  at=(4,4);
	}
	Let $M \in \pers{\ZZ}$ be the module
	\begin{equation*}
		M =	\begin{tikzpicture}[module diagram]
			\Axes(-1.5,-1.5)(5.5,5.5)
			\Grid[-1](0){5}[-1](0){5}
			\coordinatesA
			\filldraw[fill=gray, fill opacity=0.5]
			([shift={(-.5,-.5)}]R1)
			-- ([shift={(-.5,-.5)}]G1)
			-- ([shift={(-.5,-.5)}]R3)
			-- ([shift={(-.5,-.5)}]G2)
			-- ([shift={(-.5,-.5)}]R2)
			-- ([shift={(-.5,-.5)}]S1)
			-- cycle;
		\end{tikzpicture}
	\end{equation*}
	with components $M_z = \FF$ if $z$ lies in the shaded region above, and $M_z = 0$ otherwise.
	All structure maps between nonzero components of $M$ are identities.
	The sequence
	\[
	\begin{tikzcd}
		\begin{tikzpicture}[module diagram]
			\Axes(-1.5,-1.5)(5.5,5.5)
			\Grid[-1](0){5}[-1](0){5}
			\coordinatesA
			\coordinate (b) at (3,3);
			\FreeModule[violet](S1)[]{}
		\end{tikzpicture}
		\rar["{\Mtx*[r]{1 \\ -1 \\ 1}}"] \&
		\begin{tikzpicture}[module diagram]
			\Axes(-1,-1)(5.5,5.5)
			\Grid[-1](0){5}[-1](0){5}
			\coordinatesA
			\FreeModule[red](R1)[]{}
			\FreeModule[red](R2)[]{}
			\FreeModule[red](R3)[]{}
		\end{tikzpicture}
		\rar["{\Mtx*[r]{1 & 1 & 0 \\ 0 & 1 & 1}}"] \&
		\begin{tikzpicture}[module diagram]
			\Axes(-1.5,-1.5)(5.5,5.5)
			\Grid[-1](0){5}[-1](0){5}
			\coordinatesA
			\FreeModule[blue](G1)[]{}
			\FreeModule[blue](G2)[]{}
		\end{tikzpicture}
		\rar \&[-5ex]
		\begin{tikzpicture}[module diagram]
			\Axes(-1.5,-1.5)(5.5,5.5)
			\Grid[-1](0){5}[-1](0){5}
			\coordinatesA
			\filldraw[fill=gray, fill opacity=0.5]
			([shift={(-.5,-.5)}]R1)
			-- ([shift={(-.5,-.5)}]G1)
			-- ([shift={(-.5,-.5)}]R3)
			-- ([shift={(-.5,-.5)}]G2)
			-- ([shift={(-.5,-.5)}]R2)
			-- ([shift={(-.5,-.5)}]S1)
			-- cycle;
		\end{tikzpicture}\\[-2.5em]
		G_2 \& G_1 \& G_0 \& M
	\end{tikzcd}
	\]
	is an augmented free resolution of $M$.
	By \cref{thm:main-for-matrices}, the sequence
	\[
	\begin{tikzcd}
		\begin{tikzpicture}[module diagram]
			\Axes(-5.5,-5.5)(1.5,1.5)
			\Grid[-5](-4){1}[-5](-4){1}
			\coordinatesA
			\FreeModule[blue]($(1,1)-(G1)$)[]{}
			\FreeModule[blue]($(1,1)-(G2)$)[]{}
		\end{tikzpicture}
		\rar["{\Mtx*[r]{1 & 0 \\ 1 & 1 \\ 0 & 1}}"] \&
		\begin{tikzpicture}[module diagram]
			\Axes(-5.5,-5.5)(1.5,1.5)
			\Grid[-5](-4){1}[-5](-4){1}
			\coordinatesA
			\FreeModule[red]($(1,1)-(R1)$)[]{}
			\FreeModule[red]($(1,1)-(R2)$)[]{}
			\FreeModule[red]($(1,1)-(R3)$)[]{}
		\end{tikzpicture}
		\rar["{(1\; -1\; 1)}"] \&
		\begin{tikzpicture}[module diagram]
			\Axes(-5.5,-5.5)(1.5,1.5)
			\Grid[-5](-4){1}[-5](-4){1}
			\coordinatesA
			\FreeModule[violet]($(1,1)-(S1)$)[]{}
		\end{tikzpicture}
		\rar\&[-5ex]
		\begin{tikzpicture}[module diagram]
			\Axes(-5.5,-5.5)(1.5,1.5)
			\Grid[-5](-4){1}[-5](-4){1}
			\begin{scope}[x={(-1,0)}, y={(0, -1)}]
				\coordinatesA
			\end{scope}
			\filldraw[fill=gray, fill opacity=0.5]
			([shift={(.5,.5)}]R1)
			-- ([shift={(.5,.5)}]G1)
			-- ([shift={(.5,.5)}]R3)
			-- ([shift={(.5,.5)}]G2)
			-- ([shift={(.5,.5)}]R2)
			-- ([shift={(.5,.5)}]S1)
			-- cycle;
		\end{tikzpicture}
		\\[-2.5em]
		G^\dagger_0 \&
		G^\dagger_1 \&
		G^\dagger_2 \&
		M^*
	\end{tikzcd}
	\]
	is an augmented free resolution of $M^*$.
\end{example}

\section{Relation to Grothendieck and Greenlees--May duality}
\label{sec:local-duality}
\newcommand{\LL}{\mathbb{L}}
\newcommand{\RR}{\mathbb{R}}
\newcommand{\ILambda}{\mathsf{\Lambda}}
\newcommand{\Mod}[1]{#1\mhyphen\mathbf{Mod}}
\newcommand{\grMod}[1]{#1\mhyphen\mathbf{grMod}}
\newcommand{\KHom}{\mathit{Hom}}
\newcommand{\KIHom}{\mathsfit{Hom}}

In this section, give our alternate proof of \cref{thm:main}, 
%(and hence \cref{thm:local-duality}) 
which casts the result as a special case of
multigraded, derived Grothendieck local duality.
Recall that \cref{thm:main} implies our main results, \cref{thm:local-duality,thm:module-resolutions};
see \cpageref{proof:local-duality,proof:module-resolutions}.

We begin with a few contextual remarks about this form of local duality.
In the ungraded setting, while many references state an underived form of Grothendieck local duality, the stronger derived formulation is also classical, dating back at least to \textcite[278]{Hartshorne:1966}.
In the multigraded setting, the underived form of local duality is given by \textcite[Theorem~2.2.2]{GotoWatanabe:1978a}, and the derived form is implicit in their proof.
Furthermore, the derived form has been formulated explicitly in greater generality, as \emph{multigraded Greenlees--May Duality}, by \textcite[Theorem~5.3]{Miller:2002}.

In what follows, we introduce these graded duality results and discuss their relation to \cref{thm:main}.  We begin with a brief review of the construction of derived functors.
For standard definitions from ring and module theory, we refer to the textbooks \cite{Matsumura:1987,BrunsHerzog:1998}.

\subsection{Derived functors}\label{Sec: Derived functors}
Let $\mathbf{C}$, $\mathbf{D}$ be abelian categories.
Recall the construction of the derived category $\mathcal{D}(\mathbf{C})$ of $\mathbf{C}$ from \cref{sec:derived-cat}.
Denote by $q_{\mathbf{C}}\colon \Ch(\mathbf{C}) \to \mathcal{D}(\mathbf{C})$ the localization functor.
Recall that we may identify chain complexes with cochain complexes via $C^i = C_{-i}$.

\begin{definition}
	Let $G\colon \Ch(\mathbf{C}) \to \Ch(\mathbf{D})$ be an additive functor.
	\begin{enumerate}
		\item The \emph{left total derived functor} of  $G$, denoted $\LL G\colon \mathcal{D}(\mathbf{C}) \to \mathcal{D}(\mathbf{D})$,
			is the right Kan extension of $q_{\mathbf{D}}\circ G$ along $q_{\mathbf{C}}$.
			The homology $H_d(\LL G(-))$ is called the $d$th left derived functor of $G$.
		\item The \emph{right total derived functor} of $G$, denoted $\RR G\colon \mathcal{D}(\mathbf{C}) \to \mathcal{D}(\mathbf{D})$,
			is the left Kan extension of $q_{\mathbf{D}}\circ G$ along $q_{\mathbf{C}}$.
			The cohomology $H^d(\RR G(-))$ is called the $d$th right derived functor of $G$.
	\end{enumerate}
\end{definition}
In the following, we often suppress the explicit notation for the localization functors.

If $\mathbf{C}$ has enough projectives or injectives (e.g., if $\mathbf{C} = \Mod{R}$ for some ring $R$, or $\mathbf{C} = \pers{\Z^N}$),
then left or right total derived functors always exist,
and can be computed via \emph{K-projective} and \emph{K-injective} resolutions, respectively.
We explain this for left total derived functors; the dual case
of right total derived functors is completely analogous.
\phantomsection\label{def:hom-complex}For any two chain complexes $C$, $D$, let $\KHom(C, D)$ be the chain complex of abelian groups with $\KHom(C, D)_d = \prod_i \Hom(C_i, D_{i+d})$.
%Note that $H_0(\KHom(C, D))$ is the set of chain complex morphisms from $C$ to $D$ modulo chain homotopy.
A complex $P$ is said to be \emph{K-projective} if the functor $\KHom(P, -)$ sends acyclic complexes to acyclic complexes.
A \emph{K-projective resolution} of $C$ is a quasi-isomorphism $P \to C$ with $P$ K-projective.
If $\mathbf{C}$ has enough projectives, then every complex has a K-projective resolution \cite[Corollary~3.5]{Spaltenstein:1988}.
If $C$ is concentrated
in a single degree, then a K-projective resolution is a projective resolution in the usual sense \cite[Proposition~4.3]{Yekutieli:2015}.
For any K-projective resolution $P$ of $C$, 
the natural map $\LL G(C) \to G(P)$
coming from the natural transformation $\LL G \circ q_{\mathbf{C}} \Rightarrow q_{\mathbf{D}} G$
is an isomorphism in $\mathcal{D}(\mathbf{D})$ \cite[Theorem~4.2]{Yekutieli:2015}.
If $G$ is right exact and $C_d = 0$ for $d < 0$, then $H_0(\LL G(C)) \cong G(C)_0$.
Analogously, $\RR G(C)$ is computed using a K-injective resolution of~$C$.
If $G$ is left exact and $C^d = 0$ for $d < 0$, then $H^0(\RR G(C)) \cong G(C)^0$.

\begin{remark}
	The category $\Ch(\mathbf{C})$ carries two model structures, called the \emph{projective} and the \emph{injective model structure},
	which both have quasi-isomorphisms as weak equivalences.
	In the former, the cofibrant objects are the K-projectives.
	Hence, K-projective resolutions are the cofibrant replacements.
	Since quasi-isomorphisms between K-projectives are homotopy equivalences \cite[Proposition~1.4(d)]{Spaltenstein:1988},
	every functor $G\colon \Ch(\mathbf{C})\to \Ch(\mathbf{D})$ preserves quasi-isomorphisms of K-projectives.
	The fact that $\LL G(C) \cong G(P)$ for any K-projective resolution $P$ of $C$
	then follows from a general result about derived functors in the model category setting \cite[Theorem~9.3]{DwyerSpalinski:1995}.
	Analogously, $K$-injective resolutions are fibrant replacements in the injective model structure.
\end{remark}

\begin{remark}[Contravariant functors]
	\label{rmk:contravariance}
	Consider a contravariant functor $F\colon \mathbf{C}^\op \to \mathbf{C}$.	The category $\Ch(\mathbf{C}^\op)$ is canonically isomorphic to $\Ch(\mathbf{C})^{\op}$ via negation of the homological grading.
	In this setting, a K-injective resolution of a complex $B$ in $\Ch(\mathbf{C}^\op) \cong \Ch(\mathbf{C})^\op$
	is in fact a K-projective resolution of $B$ in $\Ch(\mathbf{C})$.
	Thus, to compute $\RR F(B)$, one picks a K-projective resolution (in $\Ch(\mathbf{C})$) of $B$, and analogously for $\LL F(B)$.
\end{remark}

\paragraph{Derived bifunctors of \texorpdfstring{\boldmath $\KHom$}{Hom}}
Let
\[
	G\colon \Ch(\mathbf{B}) \times \Ch(\mathbf{C}) \to \Ch(\mathbf{D})
\]
be a bifunctor.
One may canonically identify $\Ch(\mathbf{B}) \times \Ch(\mathbf{C})$ with $\Ch(\mathbf{B}\times \mathbf{C})$.
Under this identification, a quasi-isomorphism in $\Ch(\mathbf{B}\times \mathbf{C})$ is a pair of quasi-isomorphisms, and the same is true for 
a K-projective or K-injective resolution.  % in $\Ch(\mathbf{B}\times \mathbf{C})$.
Therefore, there are canonical isomorphisms of categories $\D(\mathbf{B}) \times \D(\mathbf{C}) \cong \D(\mathbf{B}\times \mathbf{C})$
and natural isomorphisms $q_{\mathbf{B}} \times q_{\mathbf{C}} \cong q_{\mathbf{B} \times \mathbf{C}}$.
This identification yields left and right derived functors 
\[
	\LL G, \RR G\colon \D(\mathbf{B}) \times \D(\mathbf{C}) \to \D(\mathbf{D}),
\]
which are in fact, respectively, the right and left Kan 
extension of $q_{\mathbf{D}} \circ G$ along $q_{\mathbf{B}} \times q_{\mathbf{C}}$.
One can compute $\LL G(B, C)$ by applying $G$ to K-projective resolutions of $B$ and $C$,
and analogously for $\RR G(B, C)$ \cite[Theorems~9.3.11 and~9.3.16]{Yekutieli:2019}.

Now consider the bifunctor
\[
	\KHom\colon \Ch(\mathbf{C})^\op \times \Ch(\mathbf{C}) \to \Ch(\mathbf{Ab})
\]
defined above.
Its $d$th right derived functor is denoted by $\Ext^d \coloneqq H^d \circ \RR\KHom$.
Since \cref{rmk:contravariance} applies to the contravariant argument of $\KHom$,
the derived functor $\RR\KHom(B, C)$ for $B, C \in \Ch(\mathbf{C})$ can be computed by picking a K-projective resolution of $B$ (in $\Ch(\mathbf{C})$) and a K-injective resolution of $C$.
In fact, it suffices to pick the respective resolution of \emph{either $B$ or $C$}:

\begin{proposition}[{\cites[Example~4.10]{Yekutieli:2015}[cf.][Theorem~2.7.2]{Weibel:2003}}]
	\label{eq:derived-hom-complex-diamond}
Given complexes $B, C \in \Ch(\mathbf{C})$, 
a K-projective resolution $p\colon P \to B$, and a K-injective resolution $i\colon C \to I$, we have natural isomorphisms
\begin{equation*}
	\RR\KHom(B,C) \cong \KHom(P,C) \cong \KHom(B,I) \cong \KHom(P,I)
\end{equation*}
in $\D(\mathbf{Ab})$.
\end{proposition}
\begin{proof}
	Functoriality yields a commutative diagram in $\D(\mathbf{Ab})$ of the form
	\begin{equation}
		\begin{tikzcd}[
			cramped,
			row sep=small
		]
			                                                  & \RR\KHom(B, C) \ar[ddl, "p^*"'] \ar[ddr, "i_*"] &                                                   \\
			                                                  & \strut                                          &                                                   \\[-6ex]
			\RR\KHom(P,C) \ar[ddr, "i_*"]                     &                                                 & \RR\KHom(B,I) \ar[ddl, "p^*"']                    \\
			\KHom(P,C) \ar[u, densely dashed]\ar[ddr, "i_*"'] &                                                 & \KHom(B,I) \ar[u, densely dashed] \ar[ddl, "p^*"] \\[-6ex]
			                                                  & \RR\KHom(P,I)                                   &                                                   \\
			                                                  & \KHom(P,I) \ar[u, densely dashed]               &
		\end{tikzcd}
	\end{equation}
	with the localization functors $q_{\mathbf{Ab}}$ and $q_{\mathbf{C}^\op \times \mathbf{C}}$ left implicit.
	Here, the dashed morphisms come from the universal natural transformation $q_{\mathbf{Ab}}\circ\KHom \Rightarrow \RR\KHom \circ q_{\mathbf{C}^\op \times \mathbf{C}}$
	obtained from $\RR\KHom$ as a left Kan extension of $q_{\mathbf{Ab}}\circ\KHom$ along $q_{\mathbf{C}^\op \times \mathbf{C}}$.
	Since localization sends quasi-isomorphisms to isomorphisms, the functoriality of $\RR\KHom$ implies that the four maps of the top diamond are isomorphisms.
	As shown in \cite[Theorem~12.2.1 and Lemma 12.2.2]{Yekutieli:2019}, the dashed and bottom diagonal maps are also isomorphisms.
	The result follows.
\end{proof}

\paragraph{The graded setting}
In the following, the notation for graded modules aligns with the notation for persistence modules from above.
Let $G$ be a finitely generated abelian group and let $R$ be a $G$-graded ring.

We write $\grMod{R}$ for the abelian category of graded $R$-modules.
Its objects are graded $R$-modules, and $\Hom(L,M)$ is given by $R$-linear morphisms $f$ with $f(L_g) \subseteq M_g$ for all $g \in G$.
For a graded module $M$, we write $M\Shift{g}$ for the graded module with $M\Shift{g}_h = M_{g+h}$.
The category $\grMod{R}$ has an \emph{internal hom-bifunctor}
\[
	\IHom\colon \grMod{R}^\op \times \grMod{R} \to \grMod{R},
\]
defined by $\IHom(L,M)_g = \Hom(L,M\Shift{g})$.

A graded $R$-module $P$ is projective if and only if $\IHom(P,-)$ is exact, and analogously for injective modules.
Analogously to $\Hom$, the bifunctor $\IHom$ gives rise to a bifunctor
\[
	\KIHom\colon \Ch(\grMod{R})^\op \times \Ch(\grMod{R}) \to \Ch(\grMod{R})
\]
with $\KIHom(C, D)_d = \prod_i \IHom(C_i, D_{i+d})$,
Its $d$th right derived functor is denoted by $\IExt$.
Note that $\IExt^d(L, M)_g = \Ext^d(L, M\Shift{g})$.  
The analogue of \cref{eq:derived-hom-complex-diamond} holds for $\RR\KIHom$ with the same proof.

\subsection{Graded Grothendieck local duality}
From now on, let $R$ be a $G$-graded local noetherian ring with graded maximal ideal $\m$.
We always assume that the homogeneous part $R_0$ equals the \emph{residue field} $\FF \coloneqq R/\m$ of $R$.
All modules are assumed to be $G$-graded.

If $M$ is a finitely generated $R$\-/module, we let
\[
	\Gamma_\m(M) \coloneqq \bigcup_{i} \Set{m \in M; \m^i m = 0} = \colim_i \Hom(R/\m^i, M)
\]
be the submodule of $M$ of elements that are annihilated by some power of $\m$.
As submodule of $M$, the module $\Gamma_\m(M)$ inherits a $G$-grading from $M$.  
The functor $\Gamma_\m\colon \grMod{R} \to \grMod{R}$ is called the \emph{$\m$-torsion functor}.
The snake lemma implies that $\Gamma_\m$ is left exact.
Its $d$th right derived functor $H^d_\m(M) \coloneqq H^d(\RR \Gamma_\m(M))$ is called the $d$th \emph{local cohomology} of $M$ supported at $\m$.
The \emph{Matlis dual} of a graded $R$-module $M$ is the graded module $M^*$ with components $(M^*)_g = \Hom_\FF(M_{-g},\FF)$.
This defines a contravariant exact functor $(-)^*\colon \grMod{R}^\op \to \grMod{R}$,
so for cochain complexes $C$ of graded $R$-modules, considering $C^*$ as a cochain complex by negating the indexing, we have
\begin{equation}
	\label{eq:Matlis-homology}
	H^d(C)^* \cong H^{-d}(C^*)
\end{equation}
Let $\SimpleModule$ denote the field $\FF$, viewed as $R$-module.
One can show that $M^* \cong \IHom(M, E(\SimpleModule))$, where the injective module $E(\SimpleModule)$ is the \emph{injective hull} of $\SimpleModule$, 
i.e., the smallest injective $R$-module containing $\SimpleModule$ as submodule \cite[Example~2.44]{Miller:2026}.
We say that $R$ is \emph{Cohen--Macaulay} if its \emph{Krull dimension} 
\[
	\dim R \coloneqq \sup \Set{\ell; 0 \subsetneq \mathfrak{p}_0 \subsetneq \dotsb \subsetneq \mathfrak{p}_\ell \subsetneq R},
\]
with the $\mathfrak{p}_i$ ranging over graded prime ideals in $R$,
equals its \emph{depth}
\[
	\depth R \coloneqq \min\Set{i; \IExt^i(\SimpleModule, R) \neq 0}.
\]
If $R$ is Cohen--Macaulay, we call $\omega_R \coloneqq H^N_\m(R)^*$ the \emph{canonical module} of $R$.
Analogous definitions apply for ungraded rings.

\begin{theorem}[Graded Grothendieck local duality {\cites[Theorem~2.2.2]{GotoWatanabe:1978a}[Corollary~5.5]{Miller:2002}[Theorem~14.5.10]{BrodmannSharp:2012}}]
	\label{eq:local-duality-graded}
	If $R$ is a graded Cohen--Macaulay ring of Krull dimension $N$ and $M$ is a finitely generated graded $R$-module, then
	\begin{equation*}
		H^{N-d}_\m(M)^* \cong \IExt^{d}(M, \omega_R).
	\end{equation*}
\end{theorem}

An analogous statement holds for ungraded rings \cites[Corollary~6.3]{Hartshorne:1966}[Theorem~3.5.8]{BrunsHerzog:1998}[Theorem~11.2.6]{BrodmannSharp:2012}.

We say that $R$ is \emph{Gorenstein} if $R$ has a finite injective resolution as graded $R$\-/module.
In that case, the Krull dimension $\dim R$ equals the \emph{injective dimension}
\[
	\injdim R \coloneqq \sup\Set{i; \IExt^i(\SimpleModule, R) \neq 0},
\]
and $\omega_R = R\Shift{g}$ for some $g \in G$ \cite[Corollary~2.2.3]{GotoWatanabe:1978a}.

\Cref{eq:local-duality-graded} admits the following derived refinement.

\begin{theorem}
	\label{eq:local-duality-graded-derived}
	If $R$ is as in \cref{eq:local-duality-graded} and $C$ is a cochain complex of graded $R$-modules, then 
	\begin{equation*}
		\RR\Gamma_\m (C)^* \cong \RR\KIHom(C, \omega_R)[N]
	\end{equation*}
	in the derived category $\mathcal{D}(\grMod{R})$.
\end{theorem}
Regarding a graded $R$-module $M$ as a cochain complex concentrated in a single degree,
we obtain \cref{eq:local-duality-graded} from \cref{eq:local-duality-graded-derived} by taking  $d$th cohomology:
\begin{multline*}
		H^{N-d}_\m(M)^*
	\overset{\text{def.}}{=}
		H^{N-d}(\RR\Gamma_\m (M))^* 
	\overset{\text{\eqref{eq:Matlis-homology}}}{\cong}
		H^{d-N}\bigl(\RR\Gamma_\m (M)^*\bigr) 
	=
		H^{d}\bigl(\RR\Gamma_\m (M)^* [-N]\bigr) \\
	\overset{\text{\cref{eq:local-duality-graded-derived}}}{\cong}
		H^d\bigl(\RR\IHom(M, \omega_R)\bigr) 
	\overset{\text{def.}}{=}
		\IExt^d(M, \omega_R).
\end{multline*}
\begin{remark}
	The ungraded version of \cref{eq:local-duality-graded-derived} appears in \cite[278]{Hartshorne:1966}.
	 Our \cref{eq:local-duality-graded} appears explicitly in the work of \textcite[Theorem~2.2.2]{GotoWatanabe:1978a}.
	 While the graded version of \cref{eq:local-duality-graded-derived} does not seem to appear in the literature,
	 one checks easily that the proof of \cite[Theorem~2.2.2]{GotoWatanabe:1978a} in fact establishes \cref{eq:local-duality-graded-derived}.
\end{remark}

\paragraph{Relation to Theorem \ref{thm:main}}
\newcommand\Cech{K}
We now explain how to obtain \cref{thm:main} from \cref{eq:local-duality-graded-derived}.
In our case	\cite[cf.][Example~14.5.18]{BrodmannSharp:2012},
the ring $R = \FF[x_1,\dotsc,x_N]$ is $\Z^N$\-/graded local with $\m = (x_1,\dotsc,x_N)$.
The ring $R$ is graded Gorenstein of dimension~$N$ \cite[Example~14.5.17]{BrodmannSharp:2012}:
That $R$ has (Krull) dimension at least $N$ can be seen from the sequence $0 \subset (x_1) \subset (x_1,x_2) \subset \dotsb$.
That $R$ has (injective) dimension at most $N$ can be seen from the injective resolution $J \coloneqq (\Omega F(\one))^*$,
where $\Omega F(\one)$ is the complex from the proof of \cref{thm:main};
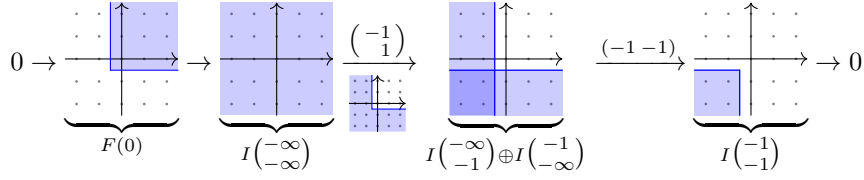
\begin{figure}
	\centering
	\newcommand{\Setup}{
		\Axes(-2.5,-2.5)(2.5,2.5)
		\Grid[-2](-1){2}[-2](-1){2}
		\coordinate (g) at (0,0);
	}
	\tikzset{every module diagram/.append style={baseline={([yshift=-1ex]current bounding box.center)}}}
	\(
		0
		\to
		\underbrace{
			\begin{tikzpicture}[module diagram]
				\Setup
				\FreeModule[blue](g){}
			\end{tikzpicture}
		}_{F(0)}
		\to
		\underbrace{
			\begin{tikzpicture}[module diagram]
				\Setup
				\InjModule<xy>[blue]($(g)-(1,1)$){}
			\end{tikzpicture}
		}_{I\tbinom{-\infty}{-\infty}}
		\xrightarrow[{
			\begin{tikzpicture}[module diagram,x=1.5mm,y=1.5mm]
				\Setup
				\FreeComplement[blue](g){}
			\end{tikzpicture}
		}]{\Mtx*[r]{-1\\1}}
		\underbrace{
			\begin{tikzpicture}[module diagram]
				\Setup
				\InjModule<x>[blue]($(g)-(1,1)$){}
				\InjModule<y>[blue]($(g)-(1,1)$){}
			\end{tikzpicture}
		}_{I\tbinom{-\infty}{-1}\oplus I\tbinom{-1}{-\infty}}
		\xto{(-1\:-1)}
		\underbrace{
			\begin{tikzpicture}[module diagram]
				\Setup
				\InjModule[blue]($(g)-(1,1)$){}
			\end{tikzpicture}
		}_{I\tbinom{-1}{-1}}
		\to 0
	\)
	\caption{The augmented injective resolution $J \coloneqq (\Omega F(\one))^*$ of $F(0)$.
		The diagram below the middle arrow visualizes the image of this morphism.
	}
	\label{fig:injres-of-free}
\end{figure}
see \cref{fig:injres-of-free} \cite[cf.][Theorem~12.11]{MillerSturmfels:2005}.
Since Krull dimension and injective dimension must coincide, we get that $R$ has dimension $N$.

Using the resolution $J$, we also see that $R$ has the dualizing module
\[
	\omega_R = H^N_\m (R)^* = H^N (\Gamma_\m (J))^* = F(\one).
\]
We will show that if $C$ is eventually vanishing, then $\RR \Gamma_\m(C) \cong C$ in $\D(\pers{\Z^N})$.
To do so, we need the following.
For a graded module $M$, recall the chain complex $\Omega M$ from \eqref{proof:main:les-1}.
Here, we view $\Omega M$ as a cochain complex concentrated in degrees $0,\dotsc,N$,
so that $\Omega^p M = \bigoplus_{Q \in \binom{\IntSet{N}}{p}} \Colim_Q M$.

For a cochain complex $C$, let $\tilde{\Omega} C$ be total complex of $\Omega C$.
Because $\Omega$ is a chain complex of colimit functors, the functor $\tilde{\Omega}\colon \Ch(\grMod{R}) \to \Ch(\grMod{R})$ is exact
and thus gives an endofunctor of $\D(\grMod{R})$.  Note that $\tilde{\Omega} C$ is the chain complex tensor product of $C$ and the \emph{\v{C}ech complex} $\Omega F(0)$,
which is a cochain complex of flat modules.

\begin{lemma}\label{Lem:Nat_Iso_Torsion_Tensor}
	The endofunctors $\RR\Gamma_\m$ and $\tilde{\Omega}$ of $\D(\grMod{R})$ are naturally isomorphic.
\end{lemma}
\begin{proof}
	\newcommand\p{\mathfrak{p}}
	We establish a natural isomorphism $\Gamma_\m(J) \cong \tilde{\Omega}J$ for every injective $J$ in $\grMod{R}$.
	Given this, it follows that $\RR\Gamma_\m \cong \RR\tilde{\Omega} \cong \tilde{\Omega}$ in $\D(\grMod{R})$;
	the first isomorphism holds because we can compute right derived functors via injective resolutions and the second isomorphism follows from exactness of $\tilde{\Omega}$.
	
	Recall that (the persistence module corresponding to) 
	an indecomposable graded injective $R$-module $J$ is of the form $J = I(z)$ for some $z \in \uZ^N$.
	Let $Q = \Set{i; z_i = \infty}$.
	Every element of $J$ is $x_i$-torsion precisely for the $i \notin Q$.
	This implies that $\Gamma_\m(J) = J$ if $Q = \emptyset$ and $\Gamma_\m(J) = 0$ otherwise.
	On the other hand, for $S \subseteq \IntSet{N}$,
	we also obtain $\Omega_S J = J$ if $S \subseteq Q$, and zero otherwise.
	Therefore,
	\[
		\tilde{\Omega}J
		= \Bigl[
			0 
			\to J 
			\to \bigoplus_{\binom{Q}{1}} J 
			\to \dotsb \to \bigoplus_{\binom{Q}{N}} J 
			\to 0
		\Bigr].
	\]
	By writing down explicit contractions, one checks easily that this complex is contractible if $Q \neq \emptyset$,
	and for $Q = \emptyset$, we obtain $\tilde{\Omega}J \cong J$.
	This shows that $\Gamma_\m(J) \cong \tilde{\Omega}J$ for all indecomposable injectives $J$, and thus for all injective modules, since the functors here are additive.  
	Naturality of this isomorphism follows because the entire proof is functorial.
\end{proof}

	\begin{proposition}
		\label{thm:derived-torsion-stabilizes-ea}
		If $C \in \Ch(\grMod{R})$ is a cochain complex such that $H^{d-N}(C),\dotsc,\allowbreak H^d(C)$ eventually vanish,
		then the natural map $H^d(\RR\Gamma_\m(C)) \to H^d(C)$ induced by the natural inclusion $\Gamma_\m C \into C$ is an isomorphism.
	\end{proposition}
	\begin{proof}
		\newcommand\EE{\mathbfit{E}}
		By \cref{Lem:Nat_Iso_Torsion_Tensor}, we have $\RR\Gamma_\m (C) \cong \tilde{\Omega} C = \tot(X)$,
		where $X$ is the double complex with $X^{pq} = \Omega^p C^q$.
		As $X$ is a double complex, it has an associated spectral sequence $\EE$ of $X$ whose zeroth page is $\EE_0^{pq} = X^{pq}$ and that converges to $H^\bullet(\tot(X)) \cong H^\bullet(\RR\Gamma_\m C)$ \cite[\S5.6]{Weibel:2003}.
		Explicitly, this means that for each $d$, the module $H^d(\tot(X))$ has an exhaustive filtration
		\[
			\dotsb \subseteq F^{p-1} \subseteq F^p \subseteq F^{p+1} \subseteq \dotsb \subseteq H^d(\tot(X))
		\]
		such that $F^p / F^{p-1} \cong \EE_\infty^{p,d-p}$.
		The first page of $\EE$ is $\EE_1^{pq} = H^q(\Omega^p C) \cong \Omega^p H^q(C)$,
		with the latter isomorphism following from exactness of the functors $\Colim_Q$.
		Since $H^q(C)$ is eventually vanishing for $q=d-N,\dotsc,d$, we get
		$\EE_r^{pq} = 0$ for $p > 0$ and $q = d-N,\dotsc,d$ in pages $r \geq 1$.
		In other words, $\EE_r^{pq}$ for $r \geq 1$ has the form
		\[
			\makeatletter
			\DeclareRobustCommand{\vdots}{\vbox{\baselineskip4\p@ \lineskiplimit\z@\hbox{.}\hbox{.}\hbox{.}}}
			\makeatother
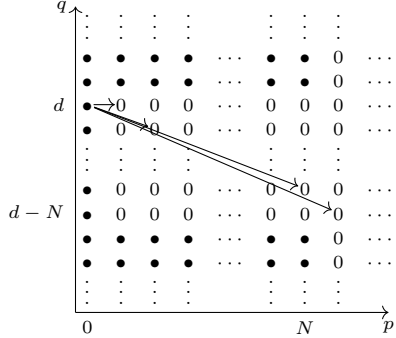
			\begin{tikzpicture}[font=\scriptsize, nodes={outer sep=0pt}]
				\matrix[matrix of math nodes, row sep=-2pt, column sep=2pt, nodes={shape=asymmetrical rectangle}] (M) {
					\vdots  & \vdots  & \vdots  & \vdots  &        & \vdots  & \vdots  & \vdots &        \\
					\bullet & \bullet & \bullet & \bullet & \cdots & \bullet & \bullet & 0      & \cdots \\
					\bullet & \bullet & \bullet & \bullet & \cdots & \bullet & \bullet & 0      & \cdots \\
					\bullet & 0       & 0       & 0       & \cdots & 0       & 0       & 0      & \cdots \\
					\bullet & 0       & 0       & 0       & \cdots & 0       & 0       & 0      & \cdots \\
					\vdots  & \vdots  & \vdots  & \vdots  &        & \vdots  & \vdots  & \vdots &        \\
					\bullet & 0       & 0       & 0       & \cdots & 0       & 0       & 0      & \cdots \\
					\bullet & 0       & 0       & 0       & \cdots & 0       & 0       & 0      & \cdots \\
					\bullet & \bullet & \bullet & \bullet & \cdots & \bullet & \bullet & 0      & \cdots \\
					\bullet & \bullet & \bullet & \bullet & \cdots & \bullet & \bullet & 0      & \cdots \\
					\vdots  & \vdots  & \vdots  & \vdots  &        & \vdots  & \vdots  & \vdots & {}     \\
				};
				\draw (M-11-1.south west)
					edge[->, "$p$"' at end] (M-11-9.south east)
					edge[->, "$q$"  at end] (M-1-1.north west);
				\path[node distance=0pt]
					node[left=of M-4-1] {$d$}
					node[left=of M-8-1] {$d-N$}
					node[below=of M-11-1] {$0$}
					node[below=of M-11-7] {$N$};
				\path[every edge/.append style={->, shorten >=-3pt, shorten <=-3pt}] (M-4-1)
					edge (M-4-2)
					edge (M-5-3)
					edge (M-7-7)
					edge (M-8-8);
			\end{tikzpicture}
		\]
		where the arrows represent the differentials $d^{0,d}_r\colon \EE^{0,d}_r \to \EE^{r,d-r+1}_r$ for $r \geq 1$ emanating from $\EE^{0,d}_r$.
		Since $\EE^{r,d-r+1}_r = 0$ for all $r \geq 1$, also the differentials $d^{0,d}_r$ are zero for $r \geq 1$ as their codomain is zero.
		Therefore, the entry $\EE_r^{0,d}$ is stationary from the first page on,
		so $\EE^{0,d}_\infty = \EE^{0,d}_1 = \Omega^0 H^d(C) = H^d(C)$,
		and $\EE^{p,d-p}_\infty = 0$ for all $p \geq 0$.
		By convergence of $\EE$, we obtain that $H^{d}(\RR\Gamma_\m C)$ has a filtration
		whose subquotients are isomorphic to the modules $\EE_\infty^{p,d-p}$ for $p \in \ZZ$.
		Of these, $\EE^{0,d}_\infty = H^d(C)$ is the only nonzero one,
		which shows that $H^d(\RR\Gamma_\m(C)) \cong H^d(\tot(X)) \cong H^d(C)$.

		Now, let $X'$ be the complex $C$, considered as double complex concentrated in the column $p = 0$.
		There is a morphism $\phi\colon X \to X'$ of double complexes.
		Let $\EE'$ be defined analogously to $\EE$ for $X'$.
		The morphism $\phi$ induces a morphism $\tilde{\phi}\colon \EE \to \EE'$ of spectral sequences,
		which converges to the morphism $H^d(\tot(\phi))\colon H^d(\Gamma_m(C)) \to H^d(C)$.
		By the above reasoning, $H^d(\tot(\phi))$ is an isomorphism.
	\end{proof}

To derive \cref{thm:main:1}, let $C$ be a chain complex with $H_d(C),\dotsc,H_{d+N}(C)$ eventually vanishing;
with cohomological indexing, this means that $H^{-d-N}(C),\dotsc,H^{-d}(C)$ eventually vanish.
As \Cref{thm:derived-torsion-stabilizes-ea} yields an isomorphism $H^{-d}(\RR\Gamma_\m C) \cong H^{-d}(C)$, we obtain
{%
	\crefname{proposition}{Prop.}{Props.}
	\crefname{theorem}{Thm.}{Thms.}
	\begin{multline}
		\label{eq:main-theorem-from-local-duality}
			H_{d+N}(\nu C)
		\overset{\text{def.}}{=}
			H_{d+N}((C^\dagger)^*)
		=
			H^{d+N}(C^\dagger)^*
		\overset{\text{\cref{eq:derived-hom-complex-diamond}}}{=}
			H^d\bigl(\RR\KIHom(C, R)[N]\bigr)^* \\
		\overset{\text{\cref{eq:local-duality-graded-derived}}}{\cong}
			H^d((\RR\Gamma_\m C)^*)^*
		\cong
			H_d(\RR\Gamma_\m C)
		\overset{\text{\cref{thm:derived-torsion-stabilizes-ea}}}{\cong}
			H_d(C),
	\end{multline}%
}%
which gives \cref{thm:main:1}.
If $H_d(C)$ eventually vanishes for all $d$,
then, again with cohomological indexing for the shift $[-N]$, we get
\begin{equation}
	\crefname{proposition}{Prop.}{Props.}
	\crefname{theorem}{Thm.}{Thms.}
	\nu C[-N]
	= (C^\dagger[N])^*
	\overset{\text{\cref{eq:derived-hom-complex-diamond}}}{=} \RR\KIHom(C, R)[N]^*
	\overset{\text{\cref{eq:local-duality-graded-derived}}}{\cong} \RR\Gamma_\m C
	\overset{\text{\cref{thm:derived-torsion-stabilizes-ea}}}{\cong} C,
\end{equation}%
which gives \cref{thm:main:2}.

\subsection{Graded Greenlees--May duality}
\renewcommand{\aa}{\mathfrak{a}}
Greenlees--May duality \cite{GreenleesMay:1992} is a far-reaching generalization of Grothendieck local duality;
see \cite{Faridian:2019} for an overview in the ungraded case.
To explain the graded case, let $R$ be a $G$-graded ring and $\aa$ be a graded ideal of $R$.
Define the \emph{graded $\aa$-adic completion} functor $\ILambda^\aa\colon \grMod{R} \to \grMod{R}$,
where $\ILambda^\aa(M)$ is the module with graded components $\ILambda^\aa(M)_g = \lim_i M_g/(\aa^i M)_g$ for $g \in G$ \cite[\S4.1]{Miller:2002}.
The derived functor $H^\aa_d(M) \coloneqq H_d(\LL \ILambda^\aa (M))$ of $\Lambda^\aa$ is the $d$th \emph{local homology} of $M$.

\begin{remark}
	\label{eq:completion}
	While the submodule $\Gamma_\aa M\subset M$ inherits a grading from $M$,
	the graded completion $\ILambda^\aa M$ is in general only a submodule of the ungraded completion $\Lambda^\aa M\coloneqq \lim_i M/\aa^i M$.
	For example, the ($\Z^N$-graded) ring $R = \FF[x_1,\dotsc,x_N]$, viewed as module over itself,
	has graded completion $\ILambda^\m(R) = R$ with respect to the maximal ideal $\m = (x_1,\dotsc,x_N)$,
	but its ungraded completion $\Lambda^\m(R) = \FF[\![x_1,\dotsc,x_N]\!]$ is the ring of formal power series in $x_1,\dotsc,x_N$
	\cites[181]{Eisenbud:1995}[Example~4.1]{Miller:2002}.
\end{remark}

Graded Greenlees--May duality theorem now states the following enriched adjunction between $\RR\Gamma_\aa$ and $\LL\ILambda^\aa$:

\begin{theorem}[Graded Greenlees--May duality {\cite[\S5-6]{Miller:2002}}]
	\label{eq:GM-graded}
	For $R$ and $\aa$ as above and for bounded cochain complexes $C$ and $D$ in $\Ch(\grMod{R})$, there is a natural isomorphism
	\begin{equation*}
		\RR \KIHom(\RR \Gamma_\aa(C), D) \cong \RR \KIHom(C, \LL \ILambda^\aa(D))
	\end{equation*}
	in $\mathcal{D}(\grMod{R})$.
\end{theorem}

For the ungraded version, see \cite{GreenleesMay:1992,Faridian:2019}.
One can recover Grothendieck local duality (\cref{eq:local-duality-graded-derived}) from \cref{eq:GM-graded} \cite[Corollary~5.5]{Miller:2002}.  We now explain this in our specific setting, where $R = \FF[x_1,\dotsc,x_N]$ and $\aa = \m$.  Let $D = I(0)$, viewed as cochain complex concentrated in degree zero.

To show that $\LL\ILambda^\m (D) \cong R$, we need the following.
We call a cochain complex $P \in \Ch(\grMod{R})$ \emph{K-flat} if the functor
\[
	P \otimes - \colon \Ch(\grMod{R}) \to \Ch(\grMod{R})
\]
is exact, where $\otimes$ denotes the tensor product of cochain complexes.
Bounded above cochain complexes of flat modules are K-flat~\cite[35]{PortaEtAl:2014}.
If $P$ is a K-flat cochain complex quasi-isomorphic to a complex $D$, then $\mathbb{L}\ILambda^\m(D)\cong \ILambda^\m(P)$ in $\mathcal D(\grMod{R})$
\cite[Proposition~3.6]{PortaEtAl:2014}\footnote{the proof for the ungraded case given in \cite{PortaEtAl:2014} applies verbatim to the graded case.}.
Thus, to compute $\mathbb{L}\ILambda^\m(I(0))$, it suffices to apply $\ILambda^\m$ to the flat resolution
\[
	\Omega F(0)\colon\qquad 0 \to F(\underbrace{\tilde{e}_{\IntSet{N}}}_{\one}) \to \dotsb \to \smashoperator{\bigoplus_{S \in \binom{\IntSet{N}}{d}}} F(\tilde{e}_S) \to \dotsb \to F(\underbrace{e_{\emptyset}}_{\bm{-\infty}})
\]
of $I(0)$, where $(\tilde{e}_S)_i = 1$ if $i \in S$, and $(\tilde{e}_S)_i = -\infty$ otherwise.
From \cref{eq:completion}, we get
\[
	\ILambda^\m (F(e_{\IntSet{N}})) = \ILambda^\m (F(\one))\cong F(\one).
\]
For $S \neq \IntSet{N}$, one verifies that $\m F(\tilde{e}_S) = F(\tilde{e}_S)$ and therefore $\ILambda^\m (F(\tilde{e}_S)) = 0$.
Together, we get that $\LL\ILambda^\m (I(0)) \cong F(\one)[-N]$ as a chain complex,
which establishes \cref{eq:local-duality-graded-derived}.

%\begin{remark}
%	One may also obtain the homology of $\LL\ILambda^\m(I(0))$ from \cite[Theorem~4.4]{Miller:2002}
%	with $n = N$, $S = \FF[x_1,\dotsc,x_N]$, $I = \m$ and $M = I(0) = F(0)^*$.
%	In this case, the Čech complex $\check{\mathcal{C}}_{\mathcal{F}}$ of \cite{Miller:2002} 
%	equals our $\Omega\Shift{\one}[N]$.
%	Recall that $\IHom(-, I(0)) \cong (-)^*$ as functors.
%	Thus, with the first isomorphism given by \cite[Theorem~4.4]{Miller:2002}, we obtain
%	\begin{multline*}
%		H_d(\LL \ILambda^\m(I(0)))
%		\cong H_d\bigl(\KHom(\Omega\Shift{\one}[N], I(0))\bigr)
%		\cong H_d(\Omega\Shift{\one}[N]^*) \\
%		\cong H_{d-N}(\Omega)^*\Shift{-\one}
%		\cong \begin{cases}
%			F(\one) & \text{if $d = N$},\\
%			0 & \text{otherwise}.
%		\end{cases}
%	\end{multline*}
%\end{remark}

\FloatBarrier
\section{The eventual acyclicity condition}
In this last section, we elaborate on the condition of \cref{thm:local-duality},
requiring that the modules $H_d(C),\dotsc,H_{d+N}(C)$ eventually vanish
in order to obtain an isomorphism $H_d(C) \cong H^{d+N}(C^\dagger)$.
Specifically, we provide an example showing that this condition is necessary,
and we show that the condition is not restrictive in practice.

\subsection{Necessity of the eventual acyclicity condition}
\label{rmk:eventual-acyclicity-necessary}
In the one-parameter case, $H^\bullet(K)$ and $H^\bullet(C_\bullet(K)^\dagger) = H^\bullet(\cup K, K)$ determine each other uniquely
even if $H_\bullet(K)$ does not eventually vanish; see \cref{thm:dSMVJ-detail}.
However, already for the two-parameter case, dropping the corresponding precondition of \cref{thm:local-duality} is not possible
even if $C = C_\bullet(K)$ comes from a simplicial filtration $K$,
and even if $H_d(C)$ and $H^{d+2}(C^\dagger)^*$ have the same Hilbert function.

\begin{example}
	\label{Ex:Boundedness_Necessary}
	\NewDocumentCommand{\SetupAxes}{s}{
	\IfBooleanTF{#1}{
		\Axes(7,7)[2,2](-1,-1)
		\Grid[0](2){6}[0](2){6}
	}{
		\Axes(-1,-1)(7,7)
		\Grid[0](2){6}[0](2){6}
	}
}
\NewDocumentCommand{\SetupK}{s}{
	\IfBooleanTF{#1}{\SetupAxes*}{\SetupAxes}%
	\coordinate (base) at (0,3);
	\coordinate (x) at (0,2);
	\coordinate (y) at (2,0);
	\coordinate (e) at (2,4);
	\coordinate (f) at (4,2);
	\coordinate (i) at (6,6);
}
\NewDocumentCommand{\SetupL}{s}{
	\IfBooleanTF{#1}{\SetupAxes*}{\SetupAxes}%
	\coordinate (base) at (0,3);
	\coordinate (x) at (0,2);
	\coordinate (y) at (2,0);
	\coordinate (z) at (x -| y);
	\coordinate (e) at (x -| y);
	\coordinate (f) at (x -| y);
	\coordinate (i) at (x -| y);
	\coordinate (j) at (6,6);
	\coordinate(g) at (2,4);
	\coordinate(h) at (4,2);
}
\NewDocumentCommand{\SetupM}{s}{
	\IfBooleanTF{#1}{\SetupAxes*}{\SetupAxes}%
	\coordinate (base) at (0,3);
	\coordinate (x) at (0,2);
	\coordinate (y) at (2,0);
	\coordinate (z) at (4,4);
	\coordinate(e) at (2,4);
	\coordinate(f) at (4,2);
	\coordinate (g) at (e -| f);
	\coordinate (h) at (g);
	\coordinate (gh) at (h);
	\coordinate (i) at (4,4);
	\coordinate (j) at (6,6);
}
\tikzset{
	Filtration/.pic = {
		\begin{scope}[
			vertex/.style={fill=black, inner sep=1pt, shape=circle},
			sans small,
			every label/.append style={label distance=2pt, inner sep=0pt},
			every edge quotes/.append style={inner sep=1pt},
		]
			\path[gray] (-2.5,-4) node [vertex] {};
			\IfSubStr{#1}{x}{\path (-2,0) coordinate (x) node [vertex, "\IfSubStr{#1}{X}{$x$}{}" {right, anchor=base}] {} ; }{}
			\IfSubStr{#1}{y}{\path (2,0) coordinate (y) node [vertex, "\IfSubStr{#1}{Y}{$y$}{}" {left, anchor=base}] {} ; }{}
			\IfSubStr{#1}{z}{\path (0,-3.5) coordinate (z) node [vertex, "\IfSubStr{#1}{Z}{$z$}{}" below] {} ; }{}
			\IfSubStr{#1}{e}{\draw (x) to[out=60, in=120, "\IfSubStr{#1}{E}{$e$}{}" {pos=0.7, above}] (y) ; }{}
			\IfSubStr{#1}{f}{\draw (x) to[out=-60, in=-120, "\IfSubStr{#1}{F}{$f$}{}" {pos=0.3,below}] (y) ;}{}
			\IfSubStr{#1}{g}{\draw (x) to[out=-90, in=180, "\IfSubStr{#1}{G}{$g$}{}"' {pos=.2}] (z) ;}{}
			\IfSubStr{#1}{h}{\draw (y) to[out=-90, in=0, "\IfSubStr{#1}{H}{$\!h$}{}" {pos=.2}] (z) ;}{}
			\IfSubStr{#1}{i}{\filldraw[opacity=0.2] (x) to[out=60, in=120] (y) to[out=-120, in=-60] (x); \IfSubStr{#1}{I}{\path ($(x)!.5!(y)$) node[anchor=center]{$i$};}{}}{}
			\IfSubStr{#1}{j}{\fill[opacity=0.2] (x) to[out=-60, in=-120] (y) to[out=-90, in=0] (z) to[in=-90, out=180] (x); \IfSubStr{#1}{J}{\path ($(z)+(0,1.25)$) node[anchor=center]{$j$};}{}}{}
			\path[] (-3,-3) rectangle (3,5);
		\end{scope}
	}
}%
	\begin{figure}
		\usetikzlibrary{quotes,matrix,calc}%
		\ExplSyntaxOn%
		\newcommand\RegExSubStr[4]{\regex_match:nnTF{#2}{#1}{#3}{#4}}%
		\ExplSyntaxOff%
		\tikzset{%
			complex/.pic = {
				\begin{scope}[fill=gray!25, n/.style={fill=black, shape=circle, inner sep=1.25pt}, nodes={inner sep=1pt, font=\scriptsize\sffamily}]
					\coordinate (1) at (-.25,0);
					\coordinate (2) at ( .25,0);
					\coordinate (3) at ( 0,-.25);
					\coordinate (4) at ( 0, .25);
					\RegExSubStr{#1}{1[0-9L]*2[0-9L]*3}{\fill (1) -- (2) -- (3) -- (1);}{}
					\RegExSubStr{#1}{1[0-9L]*2[0-9L]*4}{\fill (1) -- (2) -- (4) -- (1);}{}
					\RegExSubStr{#1}{1[0-9L]*3[0-9L]*4}{\fill (1) -- (3) -- (4) -- (1);}{}
					\RegExSubStr{#1}{2[0-9L]*3[0-9L]*4}{\fill (2) -- (3) -- (4) -- (1);}{}
					\RegExSubStr{#1}{1}{\node[n] (n1) at (1) {};}{}
					\RegExSubStr{#1}{2}{\node[n] (n2) at (2) {};}{}
					\RegExSubStr{#1}{3}{\node[n] (n3) at (3) {};}{}
					\RegExSubStr{#1}{4}{\node[n] (n4) at (4) {};}{}
					\RegExSubStr{#1}{1[0-9L]*2}{\draw[thick] (1) -- (2);}{}
					\RegExSubStr{#1}{1[0-9L]*3}{\draw[thick] (1) -- (3);}{}
					\RegExSubStr{#1}{1[0-9L]*4}{\draw[thick] (1) -- (4);}{}
					\RegExSubStr{#1}{2[0-9L]*3}{\draw[thick] (2) -- (3);}{}
					\RegExSubStr{#1}{2[0-9L]*4}{\draw[thick] (2) -- (4);}{}
					\RegExSubStr{#1}{3[0-9L]*4}{\draw[thick] (3) -- (4);}{}
					\RegExSubStr{#1}{L1}{\node["1" left]  at (1) {};}{}
					\RegExSubStr{#1}{L2}{\node["2" right] at (2) {};}{}
					\RegExSubStr{#1}{L3}{\node["3" right] at (3) {};}{}
					\RegExSubStr{#1}{L4}{\node["4" right] at (4) {};}{}
				\end{scope}
			},
			coords/.pic={
				\scoped{
					\draw[every edge/.append style={->, shorten <=-4pt}] (0,0) edge (4,0) edge(0,4);
					\foreach \i in {0,...,3}{
						\node[anchor=east] at (0,\i.5) {\i};
						\node[anchor=north] at (\i.5,0) {\i};
					}
					\draw[ystep=.9cm, xstep=.9cm, yshift=0mm] (.1mm,.1mm) grid ++(3.4cm,3.4cm);
				}
			},
			every matrix/.append style={
				ampersand replacement=\&,
				column sep={.9cm,between origins},
				row sep={.9cm,between origins},
			},
			x=.9cm,
			y=.9cm
		}%
		\subcaptionbox{$K$\label{fig:Boundedness_Necessary:filtration:K}}{%
			\begin{tikzpicture}
				\pic {coords};
				\matrix[matrix anchor=b.center, anchor=center] (m) at (.5,.5) {
					\pic{complex={1}};                 \& \pic{complex={12}};  \& \pic{complex={12 13 23}}; \& \pic{complex={123}}; \\
					\pic{complex={1}};                 \& \pic{complex={12}};  \& \pic{complex={12 13 23}}; \& \pic{complex={12 13 23}}; \\
					\pic{complex={L1}};                \& \pic{complex={1 2}}; \& \pic{complex={13 23}};    \& \pic{complex={13 23}};\\
					\coordinate (b); \pic{complex={}}; \& \pic{complex={L2}};  \& \pic{complex={2}};        \& \pic{complex={2}};\\
				};
			\end{tikzpicture}
		}\hfill
		\subcaptionbox{$L$\label{fig:Boundedness_Necessary:filtration:L}}{
			\begin{tikzpicture}
				\pic {coords};
				\matrix[matrix anchor=b.center, anchor=center] (m) at (.5,.5) {
					\pic{complex={1}};                 \& \pic{complex={123 14}};  \& \pic{complex={123 14 24}}; \& \pic{complex={123 124}}; \\
					\pic{complex={1}};                 \& \pic{complex={123 14}};  \& \pic{complex={123 14 24}}; \& \pic{complex={123 14 24}}; \\
					\pic{complex={L1}};                \& \pic{complex={12L3 L4}}; \& \pic{complex={123 24}};   \& \pic{complex={123 24}};\\
					\coordinate (b); \pic{complex={}}; \& \pic{complex={L2}};      \& \pic{complex={2}};         \& \pic{complex={2}};\\
				};
			\end{tikzpicture}
		}\hfill
		\subcaptionbox{$M$\label{fig:Boundedness_Necessary:filtration:M}}{
			\begin{tikzpicture}
				\pic {coords};
				\matrix[matrix anchor=b.center, anchor=center] (m) at (.5,.5) {
					\pic{complex={1}};                 \& \pic{complex={12}};  \& \pic{complex={123 14 24}};  \& \pic{complex={123 124}};   \\
					\pic{complex={1}};                 \& \pic{complex={12}};  \& \pic{complex={123 1L4 24}}; \& \pic{complex={123 14 24}}; \\
					\pic{complex={L1}};                \& \pic{complex={1 2}}; \& \pic{complex={1L3 23}};     \& \pic{complex={13 23}};     \\
					\coordinate (b); \pic{complex={}}; \& \pic{complex={L2}};  \& \pic{complex={2}};          \& \pic{complex={2}};         \\
				};
			\end{tikzpicture}
 		}%
		\caption[$H^\bullet(K)$ and $H^\bullet(C_\bullet(K)^\dagger)$ need not determine each other uniquely]{%
			The simplicial $\ZZ$\-/filtrations used in \cref{Ex:Boundedness_Necessary}.
			The inclusions go rightward and upward.
			Since we consider reduced homology, all complexes contain an additional vertex not shown.
		}
		\label{fig:Boundedness_Necessary:filtration}
	\end{figure}
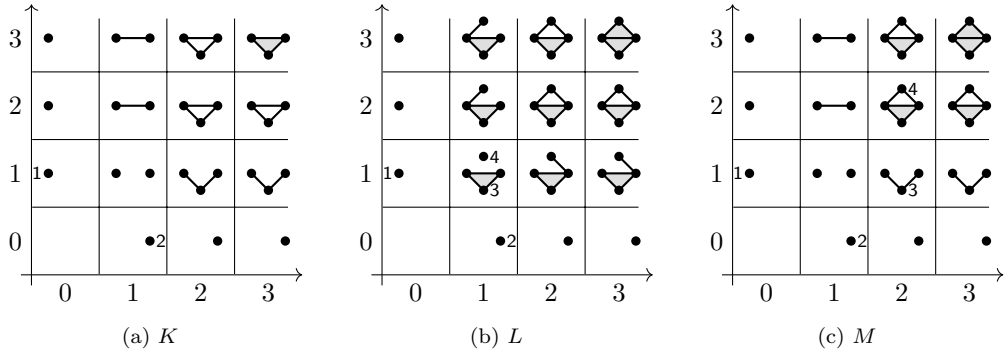
	\begin{table}
		\caption{%
			Minimal free presentations of the homology and cohomology
			of the simplicial filtrations from \cref{fig:Boundedness_Necessary:filtration}.
			For the grades, we write $\bar{\imath}$ for $-i$.
			The basis element of $C_\bullet(-)^\dagger$ dual to the simplex $\sigma$ is also denoted by $\sigma$.
		}
		\label{tab:Boundedness_Necessary}
		\centering
		$\begin{array}{@{}c|c@{\hskip6ex}c@{\hskip6ex}c@{}}
			\toprule & K & L & M \\\midrule
			H_0(-) &
			\begin{NiceMatrix}[small]
				  &      & 12    & 13+23  \\
				  &      & (02) & (20)    \\
				1 & (01) & 1    & 1       \\
				2 & (10) & 1    & 1       \\
				\CodeAfter
				\SubMatrix[{3-3}{4-4}]
			\end{NiceMatrix} &
			\begin{NiceMatrix}[small]
				    &      & 12   & 14    & 12+24  \\
				    &      & (11) & (02) & (20) \\
				1   & (01) & 1    &      &      \\
				2   & (10) & 1    &      &      \\
				1+4 & (11) &      & 1    & 1
				\CodeAfter
				\SubMatrix[{3-3}{5-5}]
			\end{NiceMatrix} &
			\begin{NiceMatrix}[small]
				  &      & 12   & 13+23 \\
				  &      & (02) & (20) \\
				1 & (01) & 1    & 1    \\
				2 & (10) & 1    & 1
				\CodeAfter
				\SubMatrix[{3-3}{4-4}]
			\end{NiceMatrix} \\[4ex]
			H^2(C_\bullet(-)^\dagger) &
			\begin{NiceMatrix}[small]
				&                  & 12               & 13               \\
				&                  & (\bar{1}\bar{2}) & (\bar{2}\bar{1}) \\
				123 & (\bar{3}\bar{3}) & 1                & 1
				\CodeAfter
				\SubMatrix[{3-3}{3-4}]
			\end{NiceMatrix} &
			\begin{NiceMatrix}[small]
				    &                  & 14               & 24               \\
				    &                  & (\bar{1}\bar{2}) & (\bar{2}\bar{1}) \\
				124 & (\bar{3}\bar{3}) & 1                & 1
				\CodeAfter
				\SubMatrix[{3-3}{3-4}]
			\end{NiceMatrix} &
			\begin{NiceMatrix}[small]
				    &                  & 14               &  12+14            & 13               \\
				    &                  & (\bar{2}\bar{2}) &  (\bar{1}\bar{2}) & (\bar{2}\bar{1}) \\
				124 & (\bar{3}\bar{3}) & 1                &                   &                  \\
				123 & (\bar{2}\bar{2}) &                  &  1                & 1
				\CodeAfter
				\SubMatrix[{3-3}{4-5}]
			\end{NiceMatrix}
			\\[2ex]
			\bottomrule
		\end{array}$
	\end{table}
	\tikzset{
		upwd/.style={trim left=(ninf), trim right=(inf)},
		downwd/.style={trim right=(ninf), trim left=(inf)},
		every module diagram/.append style={x=2.2mm, y=2.2mm, syzygy/.style={}},
		l/.style={shift={(-1,0)}},
		d/.style={shift={(0,-1)}},
		u/.style={shift={(0,1)}},
		r/.style={shift={(1,0)}}
	}
	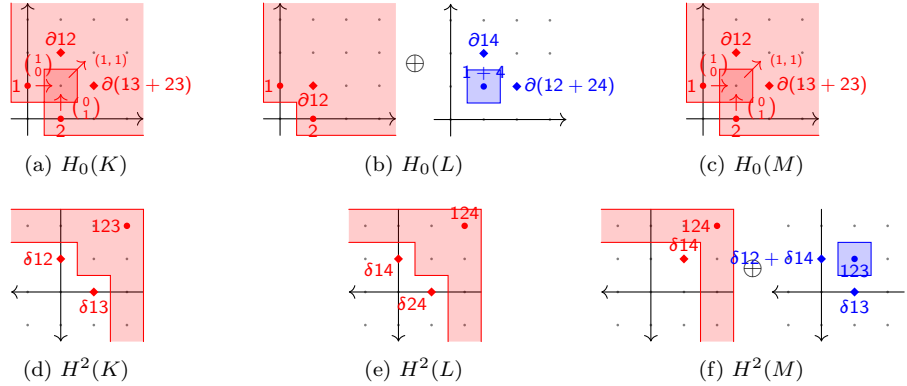
\begin{figure}
		\tikzset{every module diagram/.append style={upwd, baseline=(current bounding box.center)}}
		\subcaptionbox{$H_0(K)$}[.333\linewidth]{
			\begin{tikzpicture}[module diagram]
				\SetupK
				\begin{scope}[red]
					\fill[opacity=.2] ([l,d]x) -| ([l,d]y) -| (inf) -| cycle;
					\fill[fill opacity=0.2] ([l,d]e) rectangle ([l,d]f);
					\draw ([l]x |- inf)
						|- ([l,d]f)
						|- ([l,d]e)
						|- ([d]y -| inf);
					\draw
						(x) node[generator, "$1$" {left}]{}
						(f) node[relation, "$\partial(13+23)$" right]{}
						(e -| f) node[syzygy]{}
						(e) node[relation, "$\partial 12$" above]{}
						(y) node[generator, "$2$" below=0pt]{};
					\coordinate (J) at ([l,d]x -| y);
					\draw ($(J)+(-.5,1)$) edge[->, "$\Mtx{1\\0}$" {pos=0.15, anchor=south}] ($(J)+(.5,1)$);
					\draw ($(J)+(1,-.5)$) edge[->, "$\Mtx{0\\1}$" {pos=0.15, anchor=west}] ($(J)+(1,0.5)$);
					\coordinate (J) at ([l,d]e -| f);
					\draw ($(J)-(0.5,0.5)$) edge[->, "\tiny$({1,1})$" {at end, anchor=west}] ($(J)+(.5,.5)$);
				\end{scope}
			\end{tikzpicture}
		}\hfill
		\subcaptionbox{$H_0(L)$}[.333\linewidth]{
			\begin{tikzpicture}[module diagram]
				\SetupL
				\begin{scope}[red]
					\fill[opacity=.2] ([d,l]x) -| ([d,l]y) -| (inf) -| cycle;
					\draw ([l]x |- inf) |- ([d,l]f) |- ([d]y -| inf);
					\node[generator, "$1$" left, at=(x)]{};
					\node[generator, "$2$" below, at=(y)]{};
					\node[relation, "{$\partial 12$}"{below}, at=(f)]{};
				\end{scope}
			\end{tikzpicture}
			$\oplus$
			\begin{tikzpicture}[module diagram]
				\SetupL
				\begin{scope}[blue]
					\begin{scope}[on background layer={blue}]
						\filldraw[fill opacity=0.2] ([d,l]x -| y) rectangle ([d,l]g -| h);
					\end{scope}
					\node[at=(h), relation, "$\partial (12+24)$" right]{};
					\node[at=(g), relation, "$\partial 14$" above]{};
					\node[at=(g -| h), syzygy]{};
					\node[at=(z), generator, "$1+4$" above]{};
				\end{scope}
			\end{tikzpicture}
		}\hfill
		\subcaptionbox{$H_0(M)$}[.333\linewidth]{
			\begin{tikzpicture}[module diagram]
				\SetupM
				\begin{scope}[red]
					\fill[opacity=.2] ([l,d]x) -| ([l,d]y) -| (inf) -| cycle;
					\fill[fill opacity=0.2] ([l,d]e) rectangle ([l,d]f);
					\draw ([l]x |- inf)
						|- ([l,d]f)
						|- ([l,d]e)
						|- ([d]y -| inf);
					\draw
						(x) node[generator, "$1$" {left}]{}
						(f) node[relation, "$\partial (13+23)$" right]{}
						(e -| f) node[syzygy]{}
						(e) node[relation, "$\partial 12$" above]{}
						(y) node[generator, "$2$" below=0pt]{};
					\coordinate (J) at ([l,d]x -| y);
					\draw ($(J)+(-.5,1)$) edge[->, "$\Mtx{1\\0}$" {pos=0.15, anchor=south}] ($(J)+(.5,1)$);
					\draw ($(J)+(1,-.5)$) edge[->, "$\Mtx{0\\1}$" {pos=0.15, anchor=west}] ($(J)+(1,0.5)$);
					\coordinate (J) at ([l,d]e -| f);
					\draw ($(J)-(0.5,0.5)$) edge[->, "\tiny$({1,1})$" {at end, anchor=west}] ($(J)+(.5,.5)$);
				\end{scope}
			\end{tikzpicture}
		}\\[1em]
			\tikzset{every module diagram/.append style={downwd, baseline=(current bounding box.center)}}
			\subcaptionbox{$H^2(K)$}[.333\linewidth]{
				\begin{tikzpicture}[module diagram,trim left=(inf),trim right=(ninf)]
					\SetupK*
					\begin{scope}[red]
						\fill[opacity=.2] ([u]inf |- i) -| ([r]inf -| i) -- ([r]inf -| f) -- ([u,r]f) -| ([u,r]e) -- ([u]inf |- e) -- cycle;
						\draw ([u]inf |- i) -| ([r]inf -| i)
							([r]inf -| f) -- ([u,r]f) -| ([u,r]e) -- ([u]inf |- e);
						\node[at=(i), generator, "$123$" left]{};
						\node[at=(f), relation, "$\d 13$" below]{};
						\node[at=(e), relation, "$\d 12$" left]{};
					\end{scope}
				\end{tikzpicture}
			}\hfill
			\subcaptionbox{$H^2(L)$}[.333\linewidth]{
				\begin{tikzpicture}[module diagram]
					\SetupL*
					\begin{scope}[red]
							\fill[opacity=.2] ([u]inf |- j) -| ([r]inf -| j) -- ([r]inf -| h) -- ([u,r]h) -| ([u,r]g) -- ([u]inf |- g) -- cycle;
							\draw ([r]inf -| h) -- ([u,r]h) -| ([u,r]g) -- ([u]inf |- g)
								([u]inf |- j) -| ([r]inf -| j);
							\node[at=(j), generator, "$124$"]{};
							\node[at=(h), relation, "$\d 24$" below left]{};
							\node[at=(h-|g), syzygy]{};
							\node[at=(g), relation, "$\d 14$" below left]{};
						\end{scope}
					\end{tikzpicture}
				}\hfill
			\subcaptionbox{$H^2(M)$}[.333\linewidth]{
				\begin{tikzpicture}[module diagram,trim left=(inf),trim right=(ninf)]
					\SetupM*
					\begin{scope}[red]
						\fill[opacity=.2] ([u]inf |- j) -| ([r]inf -| j) -- ([r]inf -| h) -- ([u,r]h) -| ([u,r]g) -- ([u]inf |- g) -- cycle;
						\draw ([u]inf |- j) -| ([r]inf -| j)
							([r]inf -| h) -- ([u,r]h) -- ([u]inf |- h);
						\node[at=(j), generator, "$124$" left]{};
						\node[at=(h), relation, "${\d 14}$" above]{};
					\end{scope}
				\end{tikzpicture}
				$\oplus$
				\begin{tikzpicture}[module diagram,trim left=(inf),trim right=(ninf)]
					\SetupM*
					\begin{scope}[blue]
						\filldraw[fill opacity=0.2] ([u,r]i) rectangle ([u,r]e |- f);
						\node[at=(i), generator, "$123$" below]{};
						\node[at=(e), relation, "$\d 12+\d14$" left]{};
						\node[at=(f), relation, "$\d 13$" below]{};
					\end{scope}
				\end{tikzpicture}
			}
			\caption[]{The modules $H_0(-)$ and $H^2(C_\bullet(-)^\dagger)$ for the simplicial filtrations from \cref{fig:Boundedness_Necessary:filtration}.
				Every \tikz[module diagram, baseline=-.8ex, trim left=default, trim right=default] \node[generator]{};
				represents one generator of $H_0(-)$,
				and every \tikz[module diagram, baseline=-.8ex, trim left=default, trim right=default] \node[relation]{};
				represents a relation imposed by a (co)boundary.
				For the cohomology modules, the axes are drawn downward for better comparability,
				and the generator labeled, e.g., “$123$” really means the basis element of $C_\bullet(K)^\dagger$
				of grade $\one-g(123)$ corresponding to the simplex $123$.
			}
			\label{fig:Boundedness_Necessary:homology}
	\end{figure}
	Let $\FF = \mathbf{F}_2$ and let $K$, $L$, and $M$ be the one-critical simplicial filtrations in \cref{fig:Boundedness_Necessary:filtration}.
	We obtain the minimal free presentations of $H_0(K)$, $H_0(L)$ and $H_0(M)$
	listed in the upper row of \cref{tab:Boundedness_Necessary},
	where we decorate the rows and columns of the presentation matrices by the chains representing the generators and relations, together with their grades.
	In a similar way, we obtain minimal free presentations for the modules $H^2(C_\bullet(K)^\dagger)$, $H^2(C_\bullet(L)^\dagger)$ and $H^2(C_\bullet(M)^\dagger)$;
	see the bottom row of \cref{tab:Boundedness_Necessary}.
	For an illustration of these modules,
	see also \cref{fig:Boundedness_Necessary:homology}.
	Note that all three modules $H_0(-)$ have the same Hilbert function,
	and so do the three modules $H^2(C_\bullet(-)^\dagger)$.

	We see that $H_0(K) \cong H_0(M)$, but $H^2(C_\bullet(K)^\dagger) \not\cong H^2(C_\bullet(M)^\dagger)$ because the former is indecomposable, while the latter is not.
	Similarly, $H^2(C_\bullet(L)^\dagger) \cong H^2(C_\bullet(M)^\dagger)$, but $H_0(K) \not\cong H_0(L)$ for the same reason.
\end{example}

Similarly, one can show that $C_\bullet(K)$ and $C_\bullet(\cup K,K)$ do not determine each other uniquely
in the two-parameter case; see \cite[\S4.6]{Lenzen:2023c}.

\subsection{Making \texorpdfstring{\Boldmath$C$}{C} eventually acyclic}
\label{sec:cones}

Given a chain complex $C$ of free $\Z^N$\-/persistence modules of finite total rank, we now construct an eventually acyclic chain complex $\hat{C}$
so that $C$ is the restriction of $\hat C$ along a poset map $r\colon \Z^N\to \Z^N$.
Since restriction along $r$ defines an exact endofunctor on $\pers{\Z^N}$,
$H_d(C)$ is also the restriction along $r$ of $H_d(\hat{C})$ for any $d$,
and similarly for (minimal) free resolutions of $H_d(C)$ and $H_d(\hat C)$; see \cref{thm:acyclic-replacement-reconstruction}.
We use this to prove \cref{Cor:Combined_Result_Extended}, which is the theoretical foundation of our algorithm for computing \MFR{}s in the companion work \cite{companion-computation}.

%The upshot of this is that the requirement of \cref{thm:local-duality,thm:module-resolutions}
%that the homology of $C$ be eventually vanishing is not restrictive in practice.

\newcommand\partialh{\partial^{\mathrm{h}}}%
\newcommand\partialv{\partial^{\mathrm{v}}}%
\begin{definition}
	A \emph{double complex} is a $\Z^2$\-/indexed system $C$ of modules $C_{ij}$,
	with \emph{horizontal} and \emph{vertical} boundary maps
	$\partialh_{ij}\colon C_{ij} \to C_{i-1,j}$ and $\partialv_{ij}\colon C_{ij} \to C_{i,j-1}$,
	such that
	\begin{equation*}
		\partialh_{i-1,j} \partialh_{ij} = 0, \qquad
		\partialv_{i,j-1} \partialv_{ij} = 0, \qquad
		\partialh_{i,j-1} \partialv_{ij} = \partialv_{i-1,j}\partialh_{ij}.
	\end{equation*}
%	$\partialh_{i-1,j} \partialh_{ij} = 0$,
%	$\partialv_{i,j-1} \partialv_{ij} = 0$, and
%	$\partialh_{i,j-1} \partialv_{ij} = \partialv_{i-1,j}\partialh_{ij}$.
	The \emph{total complex} of a double complex $C$ is the chain complex $\tot C$
	with $(\tot C)_d = \bigoplus_{i+j=d} C_{ij}$
	and boundary morphisms $\partial^{\tot C}_d = \sum_{i+j=d} \partialh_{ij} + (-1)^{i}\partialv_{ij}$.
\end{definition}

\begin{example}
	For a chain complex $C$ of finite rank free modules,
	the construction $\Omega C$ from the proof of \cref{thm:main} is a double complex.
	Its total complex admits natural morphisms $C \leftarrow \tot(\Omega C)[N] \stackrel\simeq\to \nu C[N]$,
	where the latter is a quasi-isomorphism.
	\Cref{thm:main} shows that if $C$ is eventually acyclic, then also the first map is a quasi-isomorphism.
\end{example}

\begin{definition}
	The \emph{cone} of a chain complex $C$, denoted by $\cone C$,
	is the total complex of the double complex $0 \to C \xto{\id} C \to 0$,
	with the right $C$ in (horizontal) degree zero.
	Explicitly, we have $(\cone C)_d = C_d \oplus C_{d-1}$,
	and $\partial^{\cone C}_d = \begin{psmallmatrix}\smash{\partial_d} & \id \\ 0 & \smash{-\partial_{d-1}}\end{psmallmatrix}$.
	It is a contractible chain complex with a natural inclusion $C \into \cone C$;
	see \cite[Exercise~1.5.1]{Weibel:2003}.
\end{definition}

The name \enquote{cone} arises from the fact that if $K$ is a simplicial complex,
then $\cone C_\bullet(K) \cong C_\bullet(K * \{\mathrm{pt}\})$, where $*$ denotes the join of simplicial complexes.

For the remainder of this section, let  $C$ be a chain complex of free $\Z^N$\-/persistence modules of finite total rank.
The definition of $\hat{C}$ depends on a choice of vector $\zeta \in \Z^N$.
The construction is such that the restrictions of $\hat C$ and $C$ to indices $z \leq \zeta$ are equal, while the restriction of  $\hat C$ to indices $z \nleq \zeta$ is acyclic.
 We will initially consider an arbitrary choice of $\zeta$,
 in order to emphasize the functorial aspects of the construction and to clarify which properties of $\hat{C}$ hold for any $\zeta$.
 However, to recover $C$ from $\hat C$, we will eventually require that $\zeta\geq \bigvee_d \bigvee \rk C_d$;
 see \cref{thm:acyclic-replacement-reconstruction} below.

For $S \subseteq \IntSet{N}$, define the exact functor $P_S\colon \pers{\Z^N} \to \pers{\Z^N}$ by
\[
	(P_S M)_z \coloneqq \begin{cases}
		M_z & \text{if $z_i > \zeta_i$ for all $i \in S$,} \\
		0   & \text{otherwise,}
	\end{cases}
\]
with the obvious structure maps.
Note that $P_\emptyset$ is the identity functor and that for $S\ne \emptyset$, the restriction of $P_S$ to indices $z\leq \zeta$ is zero.
For $z, z' \in \Z^N$, let $z \vee_S z'$ be the vector with $(z \vee_S z')_i = \max\{z_i, z'_i\}$ if $i \in S$, and $z_i$ otherwise.
If $M$ is free, then $P_S M$ is also free, with
\begin{equation}
	\label{eq:P-rank}
	\rk P_S M = \Set{z \vee_S (\zeta+\one); z \in \rk M}.
\end{equation}
For  $S' \supseteq S$, we have canonical inclusions $c_{S,S'}\colon P_{S'} M \into P_S M$.
With the usual sign rule, these assemble into a chain complex $P M$
with $P_d M \coloneqq \bigoplus_{S \in \binom{\IntSet{N}}{d}} P_S M$.

\begin{definition}
	Let $\hat{C}$ be the total complex of the double complex $P C$.
\end{definition}
Explicitly, this means that the components of $\hat{C}$ are given by
\[
	\hat{C}_d \coloneqq \smashoperator{\bigoplus_{\emptyset\subseteq S\subseteq \IntSet{N}}} P_S C_{d-\abs{S}}.
\]
Note that there is a canonical inclusion $p\colon C \into \hat{C}$.
\begin{proposition}
	\label{thm:acyclic-replacement}
	\leavevmode
	\begin{enumerate}
		\item \label{thm:acyclic-replacement:2}
		The chain complex $\hat{C}$ is an eventually acyclic complex of free modules.
		\item \label{thm:acyclic-replacement:4} We have
		$\abs{\rk \hat{C}_d} = \sum_{i=0}^N \binom{N}{i} \abs{\rk C_{d-i}}$.
		Thus, for $N$ constant we have
		\[
			\sum_d \abs{\rk \hat{C}_d}=\mathcal{O}\Bigl(\sum_d \abs{\rk C_d}\Bigr).
		\]
	\end{enumerate}
\end{proposition}
\begin{proof}
	For $z \in \Z^N$, we write $(\hat C)_z$ for the chain complex of vector spaces
	\[
		\dotsb
		\xto{(\partial^{\hat{C}}_3)_z} (\hat{C}_2)_z
		\xto{(\partial^{\hat{C}}_2)_z} (\hat{C}_1)_z
		\xto{(\partial^{\hat{C}}_1)_z} (\hat{C}_0)_z.
	\]
	To show \proofref{thm:acyclic-replacement:2}, note that
	for $z \nleq \zeta$, the definition of $P_S$ gives
	\[
		(\hat{C})_z  =\tot \bigl(\textstyle \bigoplus_{S \in \binom{U}{\abs{U}}} (C)_z \to \dotsb \to \bigoplus_{S \in \binom{U}{0}} (C)_z \bigr)
	\]
	with $U = \Set{i; z_i \geq \zeta_i}$, where the last term is in (horizontal) degree zero.
	One checks that
	\[
		(\hat{C})_z  =\cone\bigl(\cdots (\cone C)\cdots\bigr),
	\]
	with $\abs{U}$ iterated cone operations.
	In particular, $(\hat{C})_z$ is a cone, hence acyclic.
	Therefore, $\hat{C}$ is eventually acyclic.
	Statement \proofref{thm:acyclic-replacement:4} is clear by construction.
\end{proof}

%\begin{remark}
%	As an important special case, note that if $N$ is fixed and $C$ is a chain complex of a free $\Z^N$\-/persistence modules arising as the homology of a multifiltered clique complex with $k$ vertices,
%	then \cref{thm:acyclic-replacement-reconstruction} implies that
%	\[
%	\abs{\rk \hat{C}_d} = \mathcal{O}(\abs{\rk C_d}) = \mathcal{O}(k^{d+1}).
%	\]
%\end{remark}

To recover $C$ from $\hat{C}$, we define the functor $R\colon\allowbreak \pers{\Z^N} \to \pers{\Z^N}$ by
\begin{equation*}
	(R M)_z \coloneqq M_{z \wedge \zeta},
\end{equation*}
where $\wedge$ denotes the componentwise minimum, and the structure maps $R M$ are the obvious ones.\footnote{This functor also appears in \cite[Definition~2.7]{Miller:2000}, where it is called $P_\zeta$.}
Note that $R$ is the restriction (i.e., pre-composition) along the noninjective poset map $r\colon \Z^N \to \Z^N$, $z \mapsto z \wedge \zeta$, so $R$ is exact.
Moreover, it is easily checked that for any $z\in \Z^N$, we have
\begin{equation}
	\label{thm:E-free}
	RF(z) = \begin{cases}
		F(z) & \text{if $z \leq \zeta$}, \\
		0    & \text{otherwise}.
	\end{cases}
\end{equation}

\begin{proposition}
	\label{thm:acyclic-replacement-reconstruction}
	Suppose $\zeta\geq \bigvee_d \bigvee \rk C_d$.
	Then
	\begin{enumerate}
	\item\label{thm:C=RChat} $C \cong R\hat{C}$.
	\item\label{thm:HC=RHChat} $R H_d(\hat{C}) \cong H_d(C) $ for all $d$.
	\item\label{thm:R-MFR} If $G$ is a (minimal) free resolution of $H_d(\hat{C})$, then $RG$ is a (minimal) resolution of $H_d(C)$.
	\end{enumerate}
\end{proposition}
\begin{proof}
	By construction, $C$ is obtained from $\hat{C}$ by dropping all free summands $F(z)$ from $\hat{C}$ with $z \nleq \zeta$,
	so \eqref{thm:E-free} implies that $C = R\hat{C}$, which establishes \proofref{thm:C=RChat}.
	By exactness of $R$, we get $H_d(C) = H_d(R\hat{C}) \cong RH_d(\hat{C})$, yielding \proofref{thm:HC=RHChat}.
	Exactness of $R$ also shows that $R$ sends free resolutions to free resolutions.
	Also, \eqref{thm:E-free} implies that $R$ sends homological balls either to homological balls, or to zero.
	Therefore, $R$ preserves minimal free resolutions, which establishes \proofref{thm:R-MFR}.
\end{proof}

\begin{example}
	Let $N = 2$ and $C$ be a chain complex of $\Z^2$\-/persistence modules $C_d$
	of finite total rank, and let $\zeta \geq \bigvee \rk C_d$ for all $d$.
	Then $\hat{C}$ is the chain complex of free $\Z^2$\-/persistence modules with
	\begin{equation*}
		\label{eq:cone-example}
		\hat{C}_d = C_d
			\oplus P_{\{1\}} C_{d-1}
			\oplus P_{\{2\}} C_{d-1}
			\oplus P_{\{1,2\}} C_{d-2},
		\quad
		\hat{\partial}_d = \begin{psmallmatrix}
			\partial^C_d & \id_{C_{d-1}}     & \id_{C_{d-1}}    &                          \\
			             & -\partial^C_{d-1} &                  & \phantom{-}\id_{C_{d-2}} \\
			             &                   & \partial^C_{d-1} & -\id_{C_{d-2}}           \\
			             &                   &                  & \partial^C_{d-2},
		\end{psmallmatrix}
	\end{equation*}
	 where $P_S C_d$ is free of rank $\Set{z \vee_S (\zeta+\one); z \in \rk C_d}$.
	It is eventually acyclic because
	\[
		\hat{C}_z = \begin{cases}
			C_z                 & \text{if $z \leq \zeta$},                     \\
			\cone (C_z)         & \text{if $\abs{\Set{i; z_i > \zeta_i}} = 1$}, \\
			\cone (\cone (C_z)) & \text{if $\abs{\Set{i; z_i > \zeta_i}} = 2$},
		\end{cases}
	\]
	and $\hat{C}$ satisfies $H_d(C) = R H_d(\hat{C})$.
\end{example}

%Recall from the introduction that \cref{Cor:Extended_Computational_Corollary}, which extends \cref{thm:local-duality,thm:module-resolutions} to chain complexes without eventually vanishing homology, is the theoretical foundation of the algorithm for computing \MFR{}s in our companion work \cite{companion-computation}.
%We now use \cref{thm:acyclic-replacement:3} to give a short proof of this.
\begin{proof}[Proof of \cref{Cor:Combined_Result_Extended}]
	\phantomsection\label{proof:Combined_Result_Extended}
If $C$ is a chain complex of finite total rank and $G$ is a minimal free resolution of $H^{d+N}(\hat C^\dagger)$,
then \cref{Cor:Combined_Result} implies that $G[N]{}^\dagger$ is a \MFR{} of $H_d(\hat C)$.
Thus, by \cref{thm:acyclic-replacement-reconstruction}, $R(G[N]{}^\dagger)$ is a minimal resolution of $H_d(C)$.
\end{proof}

\sloppy
\small
\printbibliography

\end{document}